\documentclass[11pt, twoside, a4paper]{amsart}

\usepackage{amsmath}     				%obligatory package
\usepackage{amssymb}   					%obligatory package
\usepackage{amsthm}     				%obligatory package
\usepackage{graphicx}    				%obligatory package

\usepackage{textcomp}    				%obligatory package
\usepackage[T1]{fontenc} 				%obligatory package
\usepackage{marvosym}    				%obligatory package (for the \Letter symbol for indicating corresponding author)
\usepackage[sc]{mathpazo}  	 			%obligatory package (font package)

\usepackage{authblk} 					%obligatory package
\usepackage[usenames]{xcolor}  			%obligatory package
\usepackage{lastpage} 					%obligatory package

\usepackage{enumerate}	 				%recommended package

\usepackage{hyperref} 					%Use this in pdflatex mode  %obligatory package

\definecolor{ForestGreen}{rgb}{0.15,0.416,0.18}
\definecolor{EgyptBlue}{rgb}{0.063,0.2,0.65}
\hypersetup{
	colorlinks=true,
	linkcolor=EgyptBlue,         
   citecolor=EgyptBlue,          
   urlcolor=ForestGreen          
}

\newtheorem{theorem}{Theorem}[section]
\newtheorem{corollary}[theorem]{Corollary}
\newtheorem{lemma}[theorem]{Lemma}
\newtheorem{proposition}[theorem]{Proposition}

\theoremstyle{definition}
\newtheorem{definition}[theorem]{Definition}
\theoremstyle{definition}
\newtheorem{remark}[theorem]{Remark}
\theoremstyle{definition}
\newtheorem{example}[theorem]{Example}

\numberwithin{equation}{section}
\numberwithin{table}{section}
\numberwithin{figure}{section}

\title{Sequences of contractions on cone metric spaces over Banach algebras and applications
to nonlinear systems of equations and systems of differential equations}

\author{\textbf{Cristian Daniel Alecsa}\footnote{Email: cristian.alecsa@math.ubbcluj.ro, cristian.alecsa@ictp.acad.ro}}
\affil{Babe\c s-Bolyai University, Mihail Kog\u alniceanu street no. 1, Cluj-Napoca RO-400084, Romania\\
Tiberiu Popoviciu Institute of Numerical Analysis, F\^ ant\^ anele street no. 57, Cluj-Napoca RO--400320, Romania}

\makeatletter
\renewcommand{\maketitle}{\bgroup\setlength{\parindent}{0pt}

\vspace{1truecm}
\begin{center}{\vbox{\titlefont\@title}}\end{center}
\vspace{0.5truecm}
\begin{center}{\@author} \end{center}

\egroup
}

\renewcommand{\@fnsymbol}[1]{%
    \ifcase#1 \or {\,\Letter\!} \or\textasteriskcentered\or \textasteriskcentered\textasteriskcentered 
    \else\@ctrerr\fi}
\makeatother

\makeatletter
\newcommand{\hbibitem}[4]{\bibitem{#1}{#2}
\def\@tempa{#3}%
\def\@tempb{#4}%
\ifx\@tempa\@empty\ifx\@tempb\@empty{}{}\else{}{\url{https://doi.org/#4}}\fi\else
{\href{http://www.ams.org/mathscinet-getitem?mr=#3}{MR#3}}\ifx\@tempb\@empty{}\else{; \url{https://doi.org/#4}}\fi\fi}
\makeatother

\newcommand*{\titlefont}{\fontsize{18}{21.6}\selectfont\textbf}

\makeatletter
\renewcommand\@author{\ifx\AB@affillist\AB@empty\AB@author\else
      \ifnum\value{affil}>\value{Maxaffil}\def\rlap##1{##1}%
    \AB@authlist\\[\affilsep]\vbox{\AB@affillist}
    \else  \AB@authors\fi\fi}
\makeatother

\makeatletter
\begin{document}

\maketitle

\pagestyle{plain}

\begin{center}
\noindent
\begin{minipage}{0.85\textwidth}\parindent=15.5pt

\medskip
%\begin{center}
%\noindent Received XX XXXXXX 201X, appeared XX XXXXXX 20XX % Please do not change this line, it will be edited by the technical editors.
%\smallskip

%\noindent Communicated by Handling Editor
%\end{center}
%\bigskip

{\small{
\noindent {\bf Abstract.} It is well known that fixed point problems of contractive-type mappings defined on cone metric spaces over Banach algebras are not equivalent to those in usual metric spaces (see \cite{RadenovicNEW} and \cite{LiuXu}). In this framework, the novelty of the present paper represents the development of some fixed point results regarding sequences of contractions in the setting of cone metric spaces over Banach algebras. Furthermore, some examples are given in order to strengthen our new concepts. Also, based on the powerful notion of a cone metric space over a Banach algebra, we present important applications to systems of differential equations and coupled functional equations, respectively, that are linked to the concept of sequences of contractions.}
\smallskip

% Please enter at most 6 keywords here with lowercase letters separated by commas.
\smallskip

% Please enter at most 5 Mathematics Subject Classification codes here. Please use 2010 classification codes, which can be found on the following link: http://www.ams.org/msc//msc2010.html.
}

\end{minipage}
\end{center}

%%%%%%%%%%%%%%%%%%%%%%%%%%%%%%%%%%%%%%%%%%

\section{Terminology and preliminary concepts}

In the present research article we try to tackle the convergence of sequences of contractions defined on cone metric spaces over Banach algebras. First of all we need to recall that F.F. Bonsall \cite{Bonsall} and S.B. Nadler Jr. \cite{Nadler} studied some stability results regarding sequences of contractions  defined on a whole metric space $(X,d)$. Furthermore, an interesting extension of the previous results was made by M. P\u acurar \cite{Pacurar}, who developed some fixed point results for the convergence of the sequence of fixed points of almost contractions. M. P\u acurar presented two interesting theorems, the first one regarding the pointwise convergence and the second one concerning uniform convergence of a sequence of almost contractions defined by the same coefficients. Now, our second aim of the present section is to remind some mathematical notions that are well established in the field of nonlinear analysis. For more information regarding these concepts, we kindly refer to \cite{Bonsall} and \cite{Pacurar}. We first present the idea of pointwise convergence.
\begin{definition}\label{D1.1} 
Let $(X,d)$ be a metric space. Also, let $T : X \to X$ and $T_{n} : X \to X$ be some given mappings for each $n \in \mathbb{N}$. By definition, the sequence $(T_{n})_{n \in \mathbb{N}}$ converges pointwise to $T$ on $X$, briefly $T_{n} \xrightarrow{p} T$, if for each $\varepsilon > 0$ and for every $x \in X$, there exists $N=N(\varepsilon,x) > 0$, such that for each $n \geq N$, we have that $d(T_{n}x,Tx) < \varepsilon.$
\end{definition}
We easily observe that in Definition \ref{D1.1}, one can replace the strict inequality $d(T_n x,T x) < \varepsilon$ by the non-strict inequality without changing the idea behind the concept of pointwise convergence. \\
Similarly, the particular notion of uniform convergence of a sequence of mappings is given as follows.
\begin{definition}\label{D1.2} 
Let $(X,d)$ be a metric space. Also, let $T : X \to X$ and $T_{n} : X \to X$ be some given mappings for each $n \in \mathbb{N}$. By definition, the sequence $(T_{n})_{n \in \mathbb{N}}$ converges uniformly to $T$ on $X$, briefly $T_{n} \xrightarrow{u} T$, if for each $\varepsilon > 0$, there exists $N=N(\varepsilon) > 0$, such that for each $n \geq N$ and for every $x \in X$, one has the following : $d(T_{n} x,T x) < \varepsilon.$
\end{definition}
Also, for a family of mappings we can briefly recall the fundamental notions of equicontinuity and uniform equicontinuity, respectively.
\begin{definition}\label{D1.3} 
Let $(X,d)$ be a metric space and $T_{n} : X \to X$ be some given mappings, for every $n \in \mathbb{N}$. The family $(T_{n})_{n \in \mathbb{N}}$ is called equicontinuous if and only if for every $\varepsilon > 0$ and for each $x \in X$, there exists $\delta=\delta(\varepsilon,x)>0$, such that for every $y \in X$ satisfying $d(x,y) < \delta$, one has that $d(T_{n} x,T_{n} y) < \varepsilon.$
\end{definition}
Now, regarding uniform equicontinuity of a family of operators, we employ the following definition.
\begin{definition}\label{D1.4} 
Let $(X,d)$ be a metric space and $T_{n} : X \to X$ be some given mappings, for every $n \in \mathbb{N}$. The family $(T_{n})_{n \in \mathbb{N}}$ is called uniformly equicontinuous if and only if for every $\varepsilon > 0$, there exists $\delta=\delta(\varepsilon)>0$, such that for every $x$ and $y$ in $X$, satisfying $d(x,y) < \delta$, one has that $d(T_{n} x,T_{n} y) < \varepsilon.$
\end{definition}
As before, one can easily replace the strict inequality with the non-strict one, such that the two definitions are equivalent to each other. \\
Now, it is time to remind that the starting point of the present research article is the paper of L. Barbet and K. Nachi. According to \cite{BarbetNachi}, the authors considered some fixed point results regarding the convergence of fixed points of contraction mappings in the regular setting of a metric space $(X,d)$. The novelty of the already mentioned paper consists on redefining pointwise and uniform convergence, respectively, but for operators defined on subsets of the whole space and not on the entire metric space $(X,d)$. Pointwise convergence was generalized by $G$-convergence and uniform convergence was extended as $H$-convergence. For the sake of completeness, we recall these two notions here.
\begin{definition}\label{D1.5} 
Let $(X,d)$ be a metric space and $X_{n}$ be nonempty subsets of $X$, for each $n \in \mathbb{N}$. Let $T_{n} : X_{n} \to X$ for every $n \in \mathbb{N}$ and $T_{\infty} : X_{\infty} \to X$ be some given mappings. By definition $T_{\infty}$ is the $G$-limit mapping of the sequence $(T_{n})_{n \in \mathbb{N}}$, whenever $(T_{n})_{n \in \mathbb{N}}$ satisfies property $(G)$, i.e.
\begin{align*}
(G): \, \forall x \in X_{\infty}, \ \exists (x_{n})_{n \in \mathbb{N}} \in \prod\limits_{n \in \mathbb{N}} X_{n}, \text{ s.t. } x_{n} \to x \text{ and } T_{n}x_{n} \to T_{\infty}x. 
\end{align*}
\end{definition}
Regarding the generalization of uniform convergence for mappings that are not defined on the whole metric space, we remind the following concept from \cite{BarbetNachi}.
\begin{definition}\label{D1.6} 
Let $(X,d)$ be a metric space and $X_{n}$ be nonempty subsets of $X$, for each $n \in \mathbb{N}$. Let $T_{n} : X_{n} \to X$ for every $n \in \mathbb{N}$ and $T_{\infty} : X_{\infty} \to X$ be some given mappings. By definition $T_{\infty}$ is the $H$-limit mapping of the sequence $(T_{n})_{n \in \mathbb{N}}$, whenever $(T_{n})_{n \in \mathbb{N}}$ satisfies property $(H)$, i.e.
\begin{align*}
(H): \, \forall (x_{n})_{n \in \mathbb{N}} \in \prod\limits_{n \in \mathbb{N}} X_{n}, \exists (y_{n})_{n \in \mathbb{N}} \subset X_{\infty} ,\text{ s.t. } d(x_{n},y_{n}) \to 0 \text{ and } d(T_{n}x_{n},T_{\infty}y_{n}) \to 0. 
\end{align*}
\end{definition}
Now, since we have reminded the basic concepts crucially important in our fixed point analysis, we make the following remark that in Theorem 2 from \cite{BarbetNachi} and in Theorem 1 from \cite{Nadler}, the authors considered the contractions to be defined on a metric space and on subset of a metric space, respectively. Moreover, they have supposed that the contractions have at least a fixed point. On the other hand, M. P\u acurar in \cite{Pacurar}
considered that the almost contractions were defined on a complete metric space and, in this case, each of them have a unique fixed point. For this, see This means that in our case it is of no importance if we consider or not the completeness of the cone metric space over the given Banach algebra.
Similarly, in [Theorem 2] of Nadler's article, that author considered the pointwise convergence of a sequence of fixed points under the assumption that the contractions are defined on a locally compact metric space $(X,d)$. Additionally, in \cite{Pacurar}, M. P\u acurar extended this result for the case of almost contractions that are defined on a complete metric space, because these mappings are not continuous so it is not properly to talk about the equicontinuity of a family of almost contractions. For this, see the observation made by M. P\u acurar before Theorem 2.6 in \cite{Pacurar}. So, in our framework of a cone metric space over a Banach algebra, it is of no loss to employ the analysis of M. P\u acurar when dealing with the completeness of such a space. Finally, for other interesting results concerning the stability of fixed points in $2$-metric spaces, stability of fixed points for sequences of $(\psi,\phi)$-weakly contractive mappings and mappings defined on an usual metric space, we let the reader follow \cite{Mishra1}, \cite{Mishra2} and \cite{SinghRusell}, respectively. \\
Now, it is time to move our focus to some articles regarding fixed point results in the setting of cone metric spaces over Banach algebras. It is well known that the fixed point theorems of contractive-type mappings defined on cone metric spaces are similar to those of the usual metric spaces, if the underlying cone is normal. These type of fixed point results were introduced in \cite{CONE}. On the other hand, H. Liu and S. Xu \cite{LiuXu} introduced the concept of cone metric spaces with Banach algebras in order to study fixed point results, replacing Banach spaces by Banach algebras and they gave an example in order to show that the fixed point results defined on this kind of spaces are non-equivalent to that of usual metric spaces. Furthermore, S. Xu and S. Radenovi\'c \cite{XuRadenovic} considered mappings defined on cone metric spaces over Banach algebras but one solid cones, without the usual assumption of normality. An interesting generalization was made by H. Huang and S. Radenovi\'c \cite{HuangRadenovic}, considering cone $b$-metric spaces over Banach algebras. They have studied common fixed points of generalized Lipschitz mappings. Also, P. Yan et. al. \cite{Yan} developed coupled fixed point theorems for mappings in the setting of cone metric spaces. Finally, the idea of replacing the Banach space by a Banach algebra was motivated by \cite{Jankovic} and \cite{Khamsi} in which some remarks about the connection between fixed point theorems for different mappings and in the case of usual normal cones of Banach spaces and usual metric spaces was given. Recently, in \cite{RadenovicNEW}, Huang et. al. studied some topological properties regarding cone metric spaces over Banach algebras. Also, they have studied some key concepts like T-stability and well-posedness regarding fixed point problems in these abstract spaces. \\
Now, at the end of this section, we are ready to review some necessary concepts and theorems regarding cone metric spaces over Banach algebras. \\
Now, considering $\mathcal{A}$ to be a Banach algebra with zero element $\theta \in \mathcal{A}$ and unit element $e \in \mathcal{A}$,we recall the notion of a cone from \cite{LiHuang}.
\begin{definition}\label{D1.7} 
A nonempty closed subset $P$ of $\mathcal{A}$ is called a cone if the following conditions hold :
\begin{align*}
& (P1) \quad \theta \text{ and } e \text{ are in } P, \\
& (P2) \quad \alpha P + \beta P \subset P, \text{ for every } \alpha,\beta \geq 0, \\
& (P3) \quad P^{2} \subseteq P, \\
& (P4) \quad P \cap (-P) = \lbrace \theta \rbrace.
\end{align*}
\end{definition}
Furthermore, we recall that $P$ is called a solid cone if $int(P) \neq \emptyset$, where $int(P)$ represent the topological interior of the set $P$. Now, as in \cite{HuangRadenovic}, one can define a partial ordering $\preceq$ with respect to the cone $P$, such as if $x$ and $y$ are in $\mathcal{A}$, then $x \preceq y$ if and only if $y-x \in P$. Also, we shall write $x \prec y$ in order to specify that $x \neq y$ and $x \preceq y$. At the same time, for $x,y \in \mathcal{A}$, we denote by $x \ll y$ the fact that $y - x \in int(P)$, based on the assumption that we will always suppose that the cone $P$ is solid. \\
From Definition 1.6 of \cite{LiHuang} and Definition 1.1 of \cite{LiuXu}, we introduce the well-known cone metric distances over the Banach algebra $\mathcal{A}$ and present some useful terminologies.
\begin{definition}\label{D1.8}
Let $X$ be a nonempty set and $d : X \times X \to \mathcal{A}$ be a mapping that satisfies the following conditions :
\begin{align*}
& (D1) \quad \theta \preceq d(x,y), \text{ for each } x,y \in X, \text{ and } d(x,y) = \theta \text{ if and only if } x = y, \\
& (D2) \quad d(x,y) = d(y,x), \text{ for each } x,y \in X, \\
& (D3) \quad d(x,y) \preceq d(x,z) + d(z,y), \text{ for every } x,y,z \in X. 
\end{align*}
Then $(X,d)$ is called a cone metric space over the Banach algebra $\mathcal{A}$.
\end{definition}
Furthermore, from \cite{XuRadenovic}, we recall the following concepts.
\begin{definition}\label{D1.9}
Let $(X,d)$ be a complete cone metric space over the Banach algebra $\mathcal{A}$. Also, let $x$ be an element of $X$ and $(x_{n})_{n \in \mathbb{N}} \subset X$ be given. Then, we have the following :
\begin{align*}
& (i) \quad (x_{n})_{n \in \mathbb{N}} \text{ converges to x, briefly } \lim\limits_{n \to \infty} x_{n} = x, \text{ if for every } c \gg \theta, \exists N = N(c) > 0, \\
& \text{ such that } d(x_{n},x) \ll c, \, \forall n \geq N. \\ 
& (ii) \quad (x_{n})_{n \in \mathbb{N}} \text{ is a Cauchy sequence }, \text{ if for every } c \gg \theta, \exists N = N(c) > 0, \\
& \text{ such that } d(x_{n},x_{m}) \ll c, \, \forall n,m \geq N. \\
& (iii) \quad (X,d) \text{ is complete if each Cauchy sequence is convergent}. 
\end{align*}
\end{definition}
In Definition \ref{D1.9}, $c \gg \theta$ represent an useful notation for $\theta \ll c$, so it lies no confusion in the rest of the present article.
Now, following the well-known Rudin's book of Functional Analysis \cite{Rudin}, for the sake of completeness, we recall the idea of the spectral radius of an element of the Banach algebra $\mathcal{A}$.
\begin{lemma}\label{L1.10}
Let $k \in \mathcal{A}$ be a given element. Then, by definition we consider the spectral radius of $k$, by 
$$ \rho(k) = \lim\limits_{n \to \infty} \| k^n \|^{\frac{1}{n}} = \inf\limits_{n \geq 1} \limits \| k^n \|^{\frac{1}{n}}. $$
If $\lambda \in \mathbb{C}$ and $\rho(k) < |\lambda|$, then the element $\lambda e - k $ is invertible. Also, one has that :
$$ (\lambda e - k )^{-1} = \sum\limits_{i=0}^{\infty} \dfrac{k^i}{\lambda^{i+1}}.$$
\end{lemma}
Now, from \cite{LiHuang}, we present some important properties regarding the spectral radius of an element of a Banach algebra $\mathcal{A}$ and some notions concerning the idea of a $c$-sequence, respectively. 
\begin{definition}\label{D1.11}
A sequence $(d_{n})_{n \in \mathbb{N}}$ from a Banach algebra $\mathcal{A}$ endowed with a solid cone $P$ is called a $c$-sequence if and only if for every $c \gg \theta$, there exists $N=N(c) \in \mathbb{N}$, for which one has $d_{n} \ll c$, for each $n > N$.
\end{definition}
Alternatively, it is easy to see that it is of no loss if we take $n \geq N$ in the above definition. Moreover, one can use, as in the case of an usual metric space, alternative definitions such as the Proposition 3.2 from \cite{XuRadenovic} when the sequence $(d_n)_{n \in \mathbb{N}}$ is from $P$. Also, we remind the fact that one can rewrite the definition of convergent sequences and Cauchy sequences respectively, using the Definition \ref{D1.11} and Definition 1.8 from \cite{LiHuang}. \\
Furthermore, we have the following properties that can be put together in a single lemma. Regarding these properties, one can follow \cite{HuangRadenovic}, \cite{Jankovic}, \cite{LiHuang} and \cite{XuRadenovic}.
\begin{lemma}\label{L1.12}
Consider $\mathcal{A}$ be a Banach algebra. Then, we have the following :
\begin{align*}
& (1) \quad \text{ if } u \preceq v \ll w \text{ or } u \ll v \preceq w, \text{ then } u \ll w, \\
& (2) \quad \text{ if } \theta \preceq u \ll c, \text{ for every } c \gg \theta, \text{ then } u = \theta,  \\
& (3) \quad \text{ if P is a cone,} (u_{n})_{n \in \mathbb{N}}, (v_{n})_{n \in \mathbb{N}} \text{ are two c-sequences in } \mathcal{A} \text{ and } \\
& \alpha,\beta \text{ are in P, then } (\alpha u_{n} + \beta v_{n})_{n \in \mathbb{N}} \text{ is also a c-sequence}, \\
& (4) \quad \text{ if P is a cone and } k \in P \text{ with } \rho(k) < 1, \text{ then } ((k)^{n})_{n \in \mathbb{N}} \text{ is a c-sequence}, \\
& (5) \quad \text{ if } k \in P, k \succeq \theta, \text{ wih } \rho(k) < 1, \text{ then } (e-k)^{-1} \succeq \theta.
\end{align*} 
\end{lemma}
On the other hand, we end this section by reminding the readers that for interesting examples of complete cone metric spaces over Banach algebras and for useful applications to functional and integral equations, we refer to \cite{HuangRadenovic}, \cite{HuangRadenovic2}, \cite{LiHuang}, and \cite{Yan}. Last, but not least, if $T$ is an operator, then by $F_T$ we denote the set of fixed points of the mapping $T$. \\
Finally, since our aim is to use the fixed point techniques in order to develop applications that have a meaningful connection with nonlinear systems of functional and differential equations, we kindly refer to \cite{HuangRadenovic} and \cite{LiHuang} for some important applications to nonlinear differential problems through fixed point results.

\section{Sequences of contractions on cone metric spaces over Banach algebras}

In the present section, we consider $\mathcal{A}$ to be a Banach algebra and $P$ to be the underlying solid cone. Our aim is to adapt in a natural way the concepts of pointwise and uniform convergence and the notions of equicontinuity for a family of mappings, respectively. First of all, we consider the definition of pointwise convergence in the framework of a cone metric space over the given Banach algebra $\mathcal{A}$.
\begin{definition}\label{D2.1} 
Let $(X,d)$ be a cone metric space over the Banach algebra $\mathcal{A}$. Also, let $T : X \to X$ and $T_{n} : X \to X$ be some given mappings for each $n \in \mathbb{N}$. By definition, the sequence $(T_{n})_{n \in \mathbb{N}}$ converges pointwise to $T$ on $X$, briefly $T_{n} \xrightarrow{p} T$, if for each $c \gg \theta$, $c \in \mathcal{A}$ and for every $x \in (X,d)$, there exists $N > 0$ that dependens on $c$ and $x$, such that for each $n \geq N$, we have that $d(T_{n} x,T x) \ll c$.
\end{definition}
In a similar way, the particular notion of uniform convergence of a sequence of mappings can be constructed as follows.
\begin{definition}\label{D2.2} 
Let $(X,d)$ be a cone metric space over the Banach algebra $\mathcal{A}$. Also, let $T : X \to X$ and $T_{n} : X \to X$ be some given mappings for each $n \in \mathbb{N}$. By definition, the sequence $(T_{n})_{n \in \mathbb{N}}$ converges uniformly to $T$ on $X$, briefly $T_{n} \xrightarrow{u} T$, if for each $c \gg \theta$, $c \in \mathcal{A}$, there exists $N > 0$ that depedens only on $c$, such that for each $n \geq N$ and for every $x \in (X,d)$, one has the following : $d(T_{n} x ,T x) \ll c.$
\end{definition}
On the other hand, for a family of mappings defined on a cone metric spaces over $\mathcal{A}$, we introduce the fundamental notions of equicontinuity and uniformly equicontinuity, respectively.
\begin{definition}\label{D2.3} 
Let $(X,d)$ be a cone metric space over the Banach algebra $\mathcal{A}$ and $T_{n} : X \to X$ be some given mappings, for every $n \in \mathbb{N}$. The family $(T_{n})_{n \in \mathbb{N}}$ is called equicontinuous if and only if for every $ c_{1} \gg \theta$, $c_{1} \in \mathcal{A}$ and for each $x \in (X,d)$, there exists $ c_{2} \gg \theta$, $c_{2} \in \mathcal{A}$ that depends on $c_{1}$ and $x$, such that for every $y \in (X,d)$ satisfying $d(x,y) \ll c_{2}$, one has that $d(T_{n} x ,T_{n} y) \ll c_{1}$, for every $n \in \mathbb{N}$.
\end{definition}
\begin{definition}\label{D2.4} 
Let $(X,d)$ be a cone metric space over the Banach algebra $\mathcal{A}$ and $T_{n} : X \to X$ be some given mappings, for every $n \in \mathbb{N}$. The family $(T_{n})_{n \in \mathbb{N}}$ is called uniformly equicontinuous if and only if for every $ c_{1} \gg \theta$, $c_{1} \in \mathcal{A}$, there exists $ c_{2} \gg \theta$, $c_{2} \in \mathcal{A}$ that depends only on $c_{1}$, such that for every $x$ and $y$ in $(X,d)$ with $d(x,y) \ll c_{2}$, one has that $d(T_{n} x,T_{n} y) \ll c_{1}$, for every $n \in \mathbb{N}$.
\end{definition}
Inspired by [Example 2.17] of \cite{HuangRadenovic} in which the authors presented a complete cone $b$-metric space over a Banach algebra with coefficient $s=2$, we are ready to present a modified version in which we have an usual complete metric space over a Banach algebra.
\begin{example}\label{E2.5}
Let's consider $\mathcal{A}$ to be set of all the matrices of the form 
$ \begin{pmatrix}
\alpha & \beta \\
0 & \alpha
\end{pmatrix}, $
where $\alpha$ and $\beta$ are from $\mathbb{R}$. On $\mathcal{A}$, we define a norm $\| \cdot \|$, such as for every matrix from $\mathcal{A}$, one has that 
$\Big\| \begin{pmatrix}
\alpha & \beta \\
0 & \alpha
\end{pmatrix}
\Big\| = |\alpha| + |\beta|$.  Also, on $\mathcal{A}$ we have the usual matrix multiplication. Moreover, one can see that 
$P = \Big\lbrace
\begin{pmatrix}
\alpha & \beta \\
0 & \alpha
\end{pmatrix} \, / \, \alpha,\beta \geq 0
\Big\rbrace$
is a nonempty solid cone on $\mathcal{A}$. Furthermore, one can verify that $\mathcal{A}$ is a Banach algebra. For the sake of completeness, we verify that the well-know triangle inequality holds under multiplication. That means that we verify that $\| A \cdot B \| \leq \| A \| \cdot \| B \|$, for every matrices $A$ and $B$, i.e. when 
$A = \begin{pmatrix}
\alpha_{1} & \beta_{1} \\
0 & \alpha_{1}
\end{pmatrix}$
and
$B = \begin{pmatrix}
\alpha_{2} & \beta_{2} \\
0 & \alpha_{2}
\end{pmatrix}$. It follows that $\| A \cdot B \| = |\alpha_{1}\alpha_{2}| + |\alpha_{1}\beta_{2} + \alpha_{2}\beta_{1}| \leq |\alpha_{1}\alpha_{2}| + |\alpha_{1}\beta_{2}| + |\alpha_{2}\beta_{1}|$. At the same time, it follows that $\| A \| \cdot \| B \| = |\alpha_{1}\alpha_{2}| + |\alpha_{1}\beta_{2}| + |\alpha_{2}\beta_{1}| + |\beta_{1}\beta_{2}|$. From all of this, it is easy to see that $\| A \cdot B \| \leq | A \| \cdot \| B \|$. \\
Now, we consider $X = [0,1]$ and define $d : X \times X \to \mathcal{A}$, such as for every $x,y \in X$, we have 
$d(x,y) = \begin{pmatrix}
|x-y| & k \cdot |x-y| \\
0 & |x-y|
\end{pmatrix},$
where $k \geq 1$. Now, we shall validate the fact that $d$ is indeed a cone metric over the given Banach algebra $\mathcal{A}$ with the identity element $e$ to be the identity matrix $I_{2}$ and the zero element $\theta$ to be the matrix with all elements $0$. For this, we consider $x,y$ and $z$ to be arbitrary elements of $X$. 
\begin{equation*}
\begin{split}
& \bullet \, \text{ We easily observe that } d(x,y) = d(y,x), \\
& \bullet \,\, d(x,y) \succeq \theta \Leftrightarrow \begin{pmatrix}
|x-y| & k \cdot |x-y| \\
0 & |x-y|
\end{pmatrix} \succeq  \begin{pmatrix}
0 & 0 \\
0 & 0
\end{pmatrix} \\
& \Leftrightarrow
\begin{pmatrix}
|x-y| & k \cdot |x-y| \\
0 & |x-y|
\end{pmatrix} -  \begin{pmatrix}
0 & 0 \\
0 & 0
\end{pmatrix} \in P \\
& \Leftrightarrow
\begin{pmatrix}
|x-y| & k \cdot |x-y| \\
0 & |x-y|
\end{pmatrix} \in P \\
& \Leftrightarrow |x-y| \geq 0.
\end{split}
\end{equation*}
\begin{equation*}
\begin{split}
\bullet \, & \text{ Taking } A:=d(x,y), B:=d(x,z) \text{ and } C:=d(z,y), \\
& \text{ we shall show that } A \leq B + C, \text{ i.e. } B+C-A \in P.
\end{split}
\end{equation*}
\begin{equation*}
\begin{split}
& \text{ This means that : } \\
& \begin{pmatrix}
|x-z| + |z-y| & k \cdot \left[ |x-z| + |z-y| \right] \\
0 & |x-z| + |z-y|
\end{pmatrix} - \begin{pmatrix}
|x-y| & k \cdot |x-y| \\
0 & |x-y|
\end{pmatrix} \in P \\
& \Longleftrightarrow
\begin{cases}
|x-z| + |z-y| \geq |x-y| \\
k \cdot \left[ |x-z| + |z-y| \right] \geq k \cdot |x-y|
\end{cases}
, \text{ which is valid.}
\end{split}
\end{equation*}
\end{example} 
Now, based on the Example \ref{E2.5}, we shall present also an example, in which we have the uniform convergence of a sequence of mappings defined on the previous cone metric space $(X,d)$ over the Banach algebra $\mathcal{A}$ given above. \\
Moreover, from now on we specify that the notation $\lim\limits_{\substack{n \to \infty \\ (\mathcal{A})}} x_{n} = \theta$ means the convergence under the Banach algebra $\mathcal{A}$, i.e. $(x_{n})_{n \in \mathbb{N}}$ is a given sequence that satisfies the fact that is a $c$-sequence. Furthermore, for a real given sequence $(y_{n})_{n \in \mathbb{N}}$ that converges to a real number $y$, we denote $\lim\limits_{\substack{n \to \infty \\ (\mathbb{R})}} y_{n} = y$. Finally, we make the observation that if we work with sequences of mappings, the latter covergence can be understood pointwise or uniformly, depending on the given context.
\begin{example}\label{E2.6}
For every $n \in \mathbb{N}$, let $f_{n} : [0,1] \to [0,1]$, such as $f_{n}(x) = \dfrac{x}{n}$, for each $x \in [0,1]$. Also, consider $f$ to be the null mapping from $[0,1]$ to $\lbrace 0 \rbrace$. Step by step, we show that $f_{n} \xrightarrow{u} f$ with respect to the cone metric $d$ from Example \ref{E2.5}. This means that for every $c \gg \theta$, there exists $N=N(c)>0$, such that for all $n \geq N$ and for each $x \in (X,d)$, we have that $d(f_{n} x,f x) \ll c$. \\
Furthermore, we observe that
$d(f_{n} x,f x) = d \left( \dfrac{x}{n},0 \right) =
\begin{pmatrix}
\Big| \dfrac{x}{n}-0 \Big| & k \cdot \Big| \dfrac{x}{n}-0 \Big| \\
0 & \Big| \dfrac{x}{n}-0 \Big|
\end{pmatrix}$. 
Furthermore, since the null element $\theta$ is the null matrix, we obviously have that
$\begin{pmatrix}
0 & 0 \\
0 & 0
\end{pmatrix} \preceq 
\begin{pmatrix}
\dfrac{x}{n} & k \cdot \dfrac{x}{n} \\
0 & \dfrac{x}{n}
\end{pmatrix}.$
First of all, we show that
$\begin{pmatrix}
\dfrac{x}{n} & k \cdot \dfrac{x}{n} \\
0 & \dfrac{x}{n}
\end{pmatrix} \preceq
\begin{pmatrix}
\dfrac{1}{n} & k \cdot \dfrac{1}{n} \\
0 & \dfrac{1}{n}
\end{pmatrix}$.
Equivalently, this means that
$\begin{pmatrix}
\dfrac{1}{n} & k \cdot \dfrac{1}{n} \\
0 & \dfrac{1}{n}
\end{pmatrix} - 
\begin{pmatrix}
\dfrac{x}{n} & k \cdot \dfrac{x}{n} \\
0 & \dfrac{x}{n}
\end{pmatrix} \in P$.
Since $\dfrac{1-x}{n} \geq 0$, because $x \leq 1$, then the above relation is valid. \\
Denoting by $A_{n} := \begin{pmatrix}
\dfrac{1}{n} & k \cdot \dfrac{1}{n} \\
0 & \dfrac{1}{n}
\end{pmatrix}$, 
we shall show that $\lim\limits_{\substack{n \to \infty \\ (\mathcal{A})}} A_{n} = \theta$. This means that $(A_{n})_{n \in \mathbb{N}}$ is a $c$-sequence, i.e. for an arbitrary $c \gg \theta$, $c \in \mathcal{A}$, there exists $N > 0$ that depends on $c$, such that for every $n \geq N$, it follows that $A_{n} \ll c$.
For this, let's consider $c \in int(P)$. Since $c$ is arbitrary, we can freely choose $c = \begin{pmatrix}
\alpha & \beta \\
0 & \alpha
\end{pmatrix}$, with $\alpha$, $\beta > 0$. We must show that there exists an index $N > 0$ that depends on $c$, i.e. $N$ depends on $\alpha$ and $\beta$, such that for all $n \geq N$, it must follow that
$\begin{pmatrix}
\dfrac{1}{n} & k \cdot \dfrac{1}{n} \\
0 & \dfrac{1}{n}
\end{pmatrix} \ll
\begin{pmatrix}
\alpha & \beta \\
0 & \alpha
\end{pmatrix}.$
This is the same as 
$\begin{pmatrix}
\alpha - \dfrac{1}{n} & \beta - k \cdot \dfrac{1}{n} \\
0 & \alpha - \dfrac{1}{n}
\end{pmatrix} \in int(P)$,
i.e. 
$\begin{cases}
\alpha > \dfrac{1}{n}, \\
\beta > k \cdot \dfrac{1}{n}.
\end{cases}$
We know that $\lim\limits_{\substack{n \to \infty \\ (\mathbb{R})}} \dfrac{1}{n} = 0$ and also $\lim\limits_{\substack{n \to \infty \\ (\mathbb{R})}} \dfrac{k}{n} = 0$. From the first limit, it follows that for the above $\alpha > 0$, there exists $N_{1}$ that depends on $\alpha$, such that for every $n \geq N_{1}$, one has that $\dfrac{1}{n} < \alpha$. Now since $n \geq N_{1}$, it follows that there indeed exists $N_{1} := \left[ \dfrac{1}{\alpha} \right] + 1$, such as $\dfrac{1}{n} \leq \dfrac{1}{N_{1}} < \alpha$.
Now, for the second limit and for the above $\beta > 0$, there exists $N_{2}$ that depends solely on $\beta$, such that for every $n \geq N_{2}$, one has that $\dfrac{k}{n} < \beta$. Now since $n \geq N_{2}$, we get that there exists $N_{2} := \left[ \dfrac{k}{\beta} \right] + 1$, such as $\dfrac{k}{n} \leq \dfrac{k}{N_{2}} < \beta$. In our analysis, we recall that the notation with the square brackets means the well-known integer part of a given number. \\
From all of this, we can find $N := \max \lbrace N_{1},N_{2} \rbrace$ that obviously depends on $\alpha$ and $\beta$, i.e. depends on $c$, such that $\alpha > \dfrac{1}{n}$ and $\beta > \dfrac{k}{n}$, respectively. This means that $(A_{n})_{n \in \mathbb{N}}$ is a $c$-sequence. So, it implies that for an arbitrary element $c \gg \theta$, $c \in \mathcal{A}$, there exists $N=N(c) > 0$, such that for every $n \in \mathbb{N}$, we have that $ A_{n} \ll c$. Using the fact that $d(f_{n} x,f x) \preceq A_{n}$ for each $x \in [0,1]$ and using $(1)$ of Lemma \ref{L1.12}, we get the desired conclusion.
\end{example}
From \cite{LiHuang}, we recall an example of a cone metric space over a Banach algebra, which will be used further in this paper.
\begin{example}\label{E2.7}
Let $\mathcal{A} = \mathbb{R}^2$. Then $\mathcal{A}$ is a Banach algebra, with the norm given by $\| (u_{1},u_{2}) \| = |u_{1}| + |u_{2}|$, for any arbitrary element $(u_{1},u_{2})$ of $\mathcal{A}$. Moreover, we have the multiplication $u \cdot v  = (u_{1}v_{1},u_{1}v_{2}+u_{2}v_{1})$, where $u=(u_{1},u_{2})$ and $v=(v_{1},v_{2})$ are given elements. Also $P = \lbrace u=(u_{1},u_{2}) \, / \, u_{1},u_{2} \geq 0 \rbrace$ is a solid cone over $\mathbb{R}^{2}$. Taking $\tilde{X} = \mathbb{R}^2$, we can define the operator $d : \tilde{X} \times \tilde{X} \to \mathcal{A}$, by $d(x,y) = (|x_1-y_1|,|x_2-y_2|)$, where $x=(x_{1},x_{2})$ and $y=(y_{1},y_{2})$. Then $(\tilde{X},d)$ is a cone metric space over $\mathbb{R}^2$.
\end{example}
We mention that if we take $X = [0,1) \times [0,1) \subset \tilde{X}$, then it is easy to see that $(X,d)$ is also a cone metric space over $\mathbb{R}^2$, where $d$ is defined in Example \ref{E2.7}. Also, based on the previous example, we shall present a sequence of mappings that converges pointwise and does not converge uniformly toward the null mapping, with respect to the cone metric $d$.
\begin{example}\label{E2.8}
For every $n \in \mathbb{N}$, let $T_{n} : [0,1) \times [0,1) \to [0,1) \times [0,1)$, defined as $T_{n}(x) = (x_{1}^{n^{2}},x_{2}^{n})$, where $x=(x_{1},x_{2}) \in [0,1)^2$. Also, we consider the null operator $T$, i.e. $T(x) = (0,0)$, where $x \in [0,1)^2$ and $T: [0,1) \times [0,1) \to \lbrace 0 \rbrace \times \lbrace 0 \rbrace \subset [0,1) \times [0,1)$. In the present example, we shall show that $T_{n} \xrightarrow{p} T$, but $T_{n} \not\xrightarrow{u} T$. \\
$\bullet$ First of all, we shall show that the sequence $(T_{n})_{n \in \mathbb{N}}$ converges pointwise to $T$ with respect to $d$. This means that for an arbitrary $c \gg \theta$, $c \in \mathcal{R}^2$ and for every $x \in [0,1)^2$, we must find an index $N > 0$ that depends on $c$ and $x$, such that for all $n \geq N$, one has $d(T_{n} x,T x) \ll c$. For this, let's consider $c = c(c_{1},c_{2})$, with $c_{1},c_{2} > 0$. Also, let $x$ to be the pair $(x_{1},x_{2})$, such that $x_{1},x_{2} \in [0,1)$. Then, it follows that :
\begin{equation*}
\begin{split}
& d(T_{n} x,T x) \ll c \\
& \Leftrightarrow (c_1,c_2) - d(T_{n} x ,T x) \in int(P) \\
& \Leftrightarrow (c_1,c_2) - d \left( \left( x_{1}^{n^{2}}, x_{2}^{n} \right) , \left( 0,0 \right) \right) \in int(P) \\
& \Leftrightarrow 
\begin{cases}
c_{1} - x_{1}^{n^{2}} > 0 \\
c_{2} - x_{2}^{n} > 0
\end{cases}
\Leftrightarrow
\begin{cases}
c_{1} > x_{1}^{n^{2}} \\
c_{2} > x_{2}^{n}.
\end{cases}
\end{split}
\end{equation*}
Now, we shall use the fact that $\lim\limits_{\substack{n \to \infty \\ (\mathbb{R})}} x_{1}^{n^{2}} = 0$ and also $\lim\limits_{\substack{n \to \infty \\ (\mathbb{R})}} x_{2}^{n} = 0$, i.e. the functions $f_{n}(x_{1}) = x_{1}^{n^{2}}$ and $g_{n}(x_{2})=x_{2}^{n}$ coverge pointwise toward $0$.
From the first limit, it follows that for $c_{1} > 0$ considered above, there exists $N_{1} = N_{1}(c_{1},x_{1}) > 0$, such that for each $n \geq N_{1}$, we have that $x_{1}^{n^{2}} \leq x_{1}^{N_{1}^{2}} < c_{1}$. Analogous, for the second limit and for $c_{2} > 0$ considered above, there exists $N_{2} = N_{2}(c_{2},x_{2}) > 0$, such that for each $n \geq N_{2}$, one has that $x_{2}^{n} \leq x_{2}^{N_{2}} < c_{2}$. Also, we mention that $N_{1}$ and $N_{2}$ ca be formally determined as in Example \ref{E2.6}. \\
From all of this, we find $N = \max \lbrace N_{1},N_{2} \rbrace$, that depends on $c_{1},c_{2},x_{1}$ and $x_{2}$ and so depend on $c$ and $x$, for which we have $d(T_{n} x,T x) \ll c$, for every $n \geq N$. \\
$\bullet$ Now, it is time to show that $(T_{n})_{n \in \mathbb{N}}$ does not converge uniformly to $T$ with respect to $d$. We know that if $T_{n} \xrightarrow{u} T$, then for every $c \gg \theta$, $c \in \mathcal{A}$, there exists $N=N(c)>0$, such that for every $n \geq N$, one has $d(T_{n} x,T x) \ll c$, for each $x \in [0,1)^2$. For this, let's consider $c=(c_{1},c_{2})$, with $c_{1},c_{2} > 0$. We know that $d(T_{n} x,T x) \ll c$ requires that $x_{1}^{n^{2}} < c_{1}$ and $x_{2}^{n} < c_{2}$ simulatenously. For all $n \in \mathbb{N}$, we can take the particular case when $x$ depends on $n$ and choose $x_{1}(n) = 5^{-1/n^2}$ and $x_{2}(n) = 3^{-1/n}$. In this manner, we obtain that $x_{1}^{n^{2}} = \dfrac{1}{5} < c_{1}$ and $x_{2}^{n} = \dfrac{1}{3} < c_{2}$. This leads to the fact that taking $c=(c_{1},c_{2})$, with $c_{1} \leq \dfrac{1}{5}$ and $c_{2} \leq \dfrac{1}{3}$, then $c \preceq \left( \dfrac{1}{5},\dfrac{1}{3} \right)$. So, for example, if we take $c = \left( \dfrac{1}{11},\dfrac{1}{8} \right)$, we get a contradiction.
\end{example}

Now, we are ready to present our main results, i.e. regarding the pointwise and uniform convergence respectively of a sequence of mappings with respect to a cone metric over a Banach algebra $\mathcal{A}$.
\begin{theorem}\label{T2.9}
Let $(X,d)$ be a cone metric space over a Banach algebra $\mathcal{A}$. Also, consider $T_{n},T : X \to X$, for each $n \in \mathbb{N}$ such that they satisfy the following assumptions : 
\begin{align*}
& (i) \ \text{for every } n \in \mathbb{N}, T_{n} \text{ has at least a fixed point, i.e. there exists } x_{n} \in T_{n}, \\    
& (ii) \ \text{ the operator T is an } \alpha-\text{contraction with respect to the cone metric } d, \text{ i.e. there exists } \alpha \in P, \\
& \text{ with } \rho(\alpha)<1, \text{ such that } d(T x,T y) \preceq \alpha d(x,y), \text{ for all } x,y \in X, \\
& (iii) T_{n} \xrightarrow{u} T \text{ as } n \to \infty, \text{ with respect to the cone metric}, \\
& (iv) (X,d) \text{ is a complete cone over the Banach algebra } \mathcal{A}.
\end{align*}
Then, following the fact that $x^{\ast}$ is the unique fixed point of the operator $T$, we have that $(d(x_{n},x^{\ast}))_{n \in \mathbb{N}}$ is a $c$-sequence.
\end{theorem}
\begin{proof}
Following \cite{XuRadenovic}, we know that there exists a unique fixed point of $T$, i.e. $x^{\ast} \in F_{T}$. Since we need a major bound for $d(x_{n},x^{\ast})$, we consider the following estimations : \\
\begin{equation*}
\begin{split}
d(x_{n},x^{\ast}) = & d(T_{n} x_{n},T x^{\ast}) \preceq \\
& d(T_{n} x_{n},T x_{n}) + d(T x_{n},T x^{\ast}) \preceq \\
& d(T_{n} x_{n},T x_{n}) + \alpha d(x_{n},x^{\ast}).
\end{split}
\end{equation*}
Using the idea of a solid cone in the Banach algebra $\mathcal{A}$, this leads to
\begin{equation*}
\begin{split}
& \alpha d(x_{n},x^{\ast}) + d(T_{n} x_{n},T x_{n}) - d(x_{n},x^{\ast}) \in P \Leftrightarrow \\
& (\alpha-e) d(x_{n},x^{\ast}) + d(T_{n} x_{n},T x_{n}) \in P \Leftrightarrow \\
& d(T_{n} x_{n},T x_{n}) - (e-\alpha) d(x_{n},x^{\ast}) \in P.
\end{split}
\end{equation*}
We specify that here we have used the fact that $(e-\alpha)$ is the opposite element of $(\alpha-e)$ in the setting of the given Banach algebra. Also. we know that $(e-\alpha)^{-1} \succeq \theta$ because $ \alpha \succeq \theta$ and that $(e-\alpha)^{-1}$ is well defined since $\rho(\alpha)<1$. Now, using the fact that $P^2 \subset P $ and multiplying by $(e-\alpha)^{-1}$, it follows that
\begin{equation*}
\begin{split}
& (e-\alpha)^{-1} d(T_{n} x_{n},T x_{n}) - d(x_{n},x^{\ast}) \in P \Leftrightarrow \\
& d(x_{n},x^{\ast}) \preceq (e-\alpha)^{-1} d(T_{n} x_{n},T x_{n}).
\end{split}    
\end{equation*}
Now, we need to show that
\begin{align*}
& \text{ for every } c \gg \theta, c \in \mathcal{A}, \text{ there exists } N_{1} \in \mathbb{N} \text{ that depends on } c, \\
& \text{ such that for every } n \geq N_{1}, \text{ one has } d(x_{n},x^{\ast}) \preceq c.
\end{align*}
For this, let's consider $c \in \mathcal{A}$, $c \gg \theta$ an arbitrary fixed element. We know that $T_{n} \xrightarrow{u} T$. This means that 
\begin{align*}
& \text{ for every } \bar{c} \gg \theta, \bar{c} \in \mathcal{A}, \text{ there exists } N_{2} \in \mathbb{N} \text{ that depends on } \bar{c}, \\
& \text{ such that for every } n \geq N_{2}, \text{ one has } d(T_{n}(x),T(x)) \ll \bar{c}, \text{ for all } x \in (X,d). 
\end{align*}
We know that $d(x_{n},x^{\ast}) \preceq (e-\alpha)^{-1} d(T_{n}(x_{n}),T(x_{n}))$. Also $(e-\alpha)^{-1} \in P$ because $\alpha \in P$. Furthermore $d(T_n x_n, T x_n) \in P \subset \mathcal{A}$, by Proposition 3.3 from \cite{XuRadenovic} and at the same time taking $x=x_n$ in the definition of uniform convergence, we get that $(d(T_n x_{n}, T x_n))_{n \in \mathbb{N}}$ is a c-sequence. This leads to the fact that $((e-\alpha)^{-1} d(T_n x_{n}, T x_n))_{n \in \mathbb{N}}$ is also a c-sequence. So, we obtain that :
\begin{align*}
& \text{ for } c \gg \theta, \ c \in \mathcal{A}, \text{ there exists } N_{2} = N_{2}(c), \text{ such that for all } n \geq N_{2}, \text{ we have } \\
& d(x_{n},x^{\ast}) \ll c.    
\end{align*}
Using (1) of Lemma \ref{L1.12}, the conclusion follows properly.
\end{proof}

Now we are ready to present our second crucial result concerning the pointwise convergence of a sequence of operators with respect to a given Banach algebra $\mathcal{A}$.
\begin{theorem}\label{T2.10}
Let $(X,d)$ be a cone metric space over a Banach algebra $\mathcal{A}$. Also, consider $T_{n},T : X \to X$, for each $n \in \mathbb{N}$ such that they satisfy the following assumptions : 
\begin{align*}
& (i) \ \text{ the operator } T_{n} \text{ is an } \alpha-\text{contraction with respect to the cone metric } d, \text{ i.e. there exists } \alpha \in P, \\
& \text{ with } \rho(\alpha)<1, \text{ such that } d(T_{n} x,T_{n} y) \preceq \alpha d(x,y), \text{ for all } x,y \in (X,d) \text{ and } n \in \mathbb{N}, \\
& (ii) \ \text{ the operator T is an } \alpha_{0}-\text{contraction with respect to the cone metric } d, \text{ i.e. there exists } \alpha_{0} \in P \\
& \text{ with } \rho(\alpha_{0})<1, \text{ such that } d(T x,T y) \preceq \alpha_{0} d(x,y), \text{ for all } x,y \in X, \\
& (iii) T_{n} \xrightarrow{p} T \text{ as } n \to \infty, \\
& (iv) (X,d) \text{ is a complete cone over the Banach algebra } \mathcal{A}.
\end{align*}
Then, following the fact that $x_{n}^{\ast}$ are the unique fixed points of the operators $T_{n}$, we have that $(d(x_{n}^{\ast},x^{\ast}))_{n \in \mathbb{N}}$ is a $c$-sequence.
\end{theorem}

\begin{proof}
From $(i)$ and $(iv)$, we obtain that for each $n \in \mathbb{N}$, there exists a unique fixed point of $T_{n}$, i.e. $x_{n}^{\ast} \in F_{T_{n}}$. Furthermore, from hyphotesis $(ii)$ and $(iv)$, it follows that there exists a unique fixed point of the operator $T$, namely $x^{\ast} \in F_{T}$. Now, in order to obtain some bounds on $d(x_{n}^{\ast},x^{\ast})$, we consider the following estimations :
\begin{equation*}
\begin{split}
d(x_{n}^{\ast},x^{\ast}) = & d(T_{n} x_{n}^{\ast},T x^{\ast}) \preceq \\
& d(T_{n} x_{n}^{\ast},T_{n} x^{\ast}) + d(T_{n} x^{\ast},T x^{\ast}) \preceq \\
& \alpha d(x_{n}^{\ast},x^{\ast}) + d(T_{n} x^{\ast},T x^{\ast}) \Leftrightarrow.
\end{split}
\end{equation*}
This leads to the following inequalities with respect to the solid cone $P$ of the Banach algebra $\mathcal{A}$ :
\begin{equation*}
\begin{split}
& d(T_{n} x^{\ast}, T x^{\ast}) + \alpha d(x_{n}^{\ast},x^{\ast}) - d(x_{n}^{\ast},x^{\ast}) \in P \Leftrightarrow \\
& d(T_{n} x^{\ast}, T x^{\ast}) + (\alpha - e ) d(x_{n}^{\ast},x^{\ast}) \in P \Leftrightarrow \\
& d(T_{n} x^{\ast}, T x^{\ast}) - (e-\alpha) d(x_{n}^{\ast},x^{\ast}) \in P.
\end{split}    
\end{equation*}
From $\alpha \in P$ and by the fact that $\rho(\alpha)<1$, it follows that there exist $(e-\alpha)^{-1} \in P$. Multiplying by $(e-\alpha)^{-1}$ and using the fact that $P^2 \subseteq P$, we have that
$$ (e-\alpha)^{-1} d(T_{n} x^{\ast}, T x^{\ast}) - d(x_{n}^{\ast},x^{\ast}) \in P.$$
We obtain the following : 
$$ d(x_{n}^{\ast},x^{\ast}) \preceq (e-\alpha)^{-1} d(T_{n} x^{\ast},T x^{\ast}).$$
Now, one can observe that $(e-\alpha)^{-1} = \sum\limits_{i=0}^{\infty} \alpha^i$. Since $\alpha$ and $e$ are in $P$, by induction one can prove that $\alpha^i \in P$, for every $i \geq 0$. So $(e-\alpha)^{-1} \in P$. Now, we want to show that :
\begin{align*}
& \text{ for every } c \gg \theta, \ c \in \mathcal{A}, \text{ there exists } N_{2} \in \mathbb{N} \text{ that depends on } c, \\
& \text{ such that for all } n \geq N_{2}, \text{ we have } d(x_{n}^{\ast},x^{\ast}) \preceq c.     
\end{align*}
Now, from the fact that $T_{n} \xrightarrow{p} T$, it follows that
\begin{align*}
& \text{ for every } \bar{c} \gg \theta, \ \bar{c} \in \mathcal{A} \text{ and for } x \in (X,d), \text{ there exists } N_{2} \in \mathbb{N} \text{ that depends on } \bar{c} \text{ and } x, \\
& \text{ such that for all } n \geq N_{2}, \text{ we have that } d(T_{n} x,T x) \ll \bar{c}.  
\end{align*}
Now, taking $x=x^{\ast}$ fixed, we get that $(d(T_n x^{\ast},T x^{\ast}))$ is a c-sequence. Also, since $(e-\alpha)^{-1} \in P$, by Proposition 3.3 of \cite{XuRadenovic}, it follows that $(d(x_{n}^{\ast},x^{\ast}))_{n \in \mathbb{N}}$ is also a c-sequence. This reasoning can be done as in the proof of Theorem \ref{T2.9} and this completes our proof. \\
We observe that $x^{\ast}$ is fixed from the beginning, so it does not influence the rank $N_{2}$ from the definition of a $c$-sequence. This means that our conclusion is well defined. Finally, as in Theorem \ref{T2.9}, using Proposition 3.2 of \cite{XuRadenovic}, it follows also that $d(x_{n}^{\ast},x^{\ast}) \ll c$ and the proof is over.
\end{proof}
Now, as in Theorem \ref{T2.9}, one can observe that we can use an equivalent definition of pointwise convergence using non-strict inequalities, and this does not influence the obtained results.

\section{(G)-convergence and (H)-convergence}

Following \cite{BarbetNachi}, our aim of the present section is to extend the concepts of (G)-convergence and (H)-convergence, respectively for sequences of operators that have different domains, from the case of usual metric spaces to the case of cone metric spaces over a Banach algebra. We begin by extending some notions regarding these two types of convergence from metric spaces to cone metric spaces. The first concept concerns an extension of the well-known pointwise convergence, but for operators that do not have the same domain of definition.

\begin{definition}\label{D3.1}
Let $X_n, X_\infty$ be subsets of $X$, where $(X,d)$ is a cone metric space (not necessarily complete) over a given Banach algebra $\mathcal{A}$. Also, let's consider for each $n \in \mathbb{N}$ some operators $T_n : X_n \to X$ and $T_\infty : X_\infty \to X$. By definition, $T_\infty$ is a (G)-limit of the sequence $(T_n)_{n \in \mathbb{N}}$, when the family of mappings $(T_n)_{n \in \mathbb{N}}$ satisfies the following property :
\begin{equation*}
\begin{split}
(G) : & \ \text{ for each } x \in X_\infty, \text{ there exists a sequence } (x_n)_{n \in \mathbb{N}}, \text{ with } x_n \in X_n \ (n \in \mathbb{N}), \text{ such that : } \\
& (d(x_n,x))_{n \in \mathbb{N}} \text{ is a c-sequence and } (d(T_n x_n,T_\infty x))_{n \in \mathbb{N}} \text{ is also a c-sequence}.    
\end{split}
\end{equation*}
\end{definition}
Now, the second definition of the present section concerns a generalization of the uniform convergence, but for mappings that do not have the same domain.

\begin{definition}\label{D3.2}
Let $X_n, X_\infty$ be subsets of $X$, where $(X,d)$ is a cone metric space (not necessarily complete) over a given Banach algebra $\mathcal{A}$. Also, let's consider for each $n \in \mathbb{N}$ some operators $T_n : X_n \to X$ and $T_\infty : X_\infty \to X$. By definition, $T_\infty$ is a (H)-limit of the sequence $(T_n)_{n \in \mathbb{N}}$, when the family of mappings $(T_n)_{n \in \mathbb{N}}$ satisfies the following property :
\begin{equation*}
\begin{split}
(H) : & \ \text{ for each sequence } (x_n)_{n \in \mathbb{N}}, \text{ with } x_n \in X_n, \text{ for every } n \in \mathbb{N} , \\
& \text{ there exists a sequence } (y_n)_{n \in \mathbb{N}} \subset X_\infty, \text{ such that : } \\
& (d(x_n,y_n))_{n \in \mathbb{N}} \text{ is a c-sequence and } (d(T_n x_n,T_\infty y_n))_{n \in \mathbb{N}} \text{ is also a c-sequence}.    
\end{split}
\end{equation*}
\end{definition}

Our first result from this section concerns the fact that the (H)-limit of a sequence of operators is also a (G)-limit, under suitable circumstances. Moreover, since we need the idea of continuity of an operator, we can employ two definitions : an extension of the definition of continuity from the case of metric spaces to the case of cone metric spaces over Banach algebras and the second one the idea of sequential continuity (for this see (iii) of Definition 2.1 from \cite{LiHuang}). Namely, we have the following remark.

\begin{remark}\label{R3.3}
If $(X,d)$ is a cone metric space over a Banach algebra $\mathcal{A}$, then : \\
a) An operator $T$ is continuous in $x_0 \in (X,d)$ if and only if for each $c \gg \theta$, $c \in \mathcal{A}$, there exists $\bar{c} \in P$ that depends on $c$, such that for every $x \in (X,d)$, satisfying $d(x,x_0) \ll \bar{c}$, one has that $d(T(x),T(x_0)) \ll c$. Moreover, the operator $T$ is continuous if it is continuous at every point of it's domain. \\
b) An operator $T$ is sequential continuous if for every sequence $(y_n)_{n \in \mathbb{N}}$ convergent to $x \in X$, i.e. satisfying $(d(y_n,x))_{n \in \mathbb{N}}$ is a c-sequence, then $(d(T y_n,T x))_{n \in \mathbb{N}}$ is also a c-sequence.
\end{remark}

\begin{proposition}\label{P3.4}
Let $(X,d)$ be a cone metric space over a given Banach algebra $\mathcal{A}$. Also, for each $n \in \mathbb{N}$, let $X_n$ be some nonempty subsets of $X$. Also, consider another nonempty subset of $X$, namely $X_\infty$. Furthermore, suppose that the following conditions are satisfied :
\begin{align*}
& (i) \ \text{ if } x \in X_\infty, \text{ then there exists } (x_{n})_{n \in \mathbb{N}}, \text{ with } x_n \in X_n \text{ for every } n \in \mathbb{N}, \\
& \text{ such that } (d(x_n,x))_{n \in \mathbb{N}} \text{ is a c-sequence}, \\
& (ii) \ T_\infty : X_\infty \to X \text{ is sequential continuous}, \\
& (iii) \ T_\infty \text{ is a (H)-limit for the family } (T_n).
\end{align*}
Then, $T_\infty$ is a (G)-limit for the family $(T_n)$. 
\end{proposition}

\begin{proof}
Let $x \in X_\infty$. Then, from (i), we find a sequence $(x_n)_{n \in \mathbb{N}}$, with $x_n \in X_n$ for $n \in \mathbb{N}$, such that $(d(x_n,x))_{n \in \mathbb{N}}$ is a c-sequence. From (iii), we obtain that there exists a sequence $(y_n)_{n \in \mathbb{N}}$, such that $(d(y_n,x_n))_{n \in \mathbb{N}}$ and $(d(T_n x_n,T_\infty y_n))_{n \in \mathbb{N}}$ are c-sequences. At the same time, we need to show that for an arbitrary $x$ from $X_\infty$, there exists a sequence $(z_n)_{n \in \mathbb{N}}$, such that $(d(z_n,x))_{n \in \mathbb{N}}$ and $(d(T_n z_n,T_\infty x))_{n \in \mathbb{N}}$ are c-sequences. We shall show that $z_n = x_n$, for every $n \in \mathbb{N}$. So, we have that 
$$ d(y_n,x) \preceq d(y_n,x_n) + d(x_n,x). $$
Also, since $(d(y_n,x))_{n \in \mathbb{N}}$ and $(d(x_n,x))_{n \in \mathbb{N}}$ are c-sequences, then we obtain that the right hand side is also a c-sequence, i.e.
\begin{align*}
& \text{ for each } c \gg \theta, \ c \in \mathcal{A}, \text{ there exists } N = N(c) \in \mathbb{N}, \text{ such that for all } n \geq N, \\
& \text{ we have that } (d(y_n,x)) \preceq d(y_n,x_n) + d(x_n,x) \ll c, \text{ so } (d(y_n,x))_{n \in \mathbb{N}} \text{ is a c-sequence}.  
\end{align*}
From Remark \ref{R3.3}, since $(d(y_n,x))_{n \in \mathbb{N}}$ is a c-sequence, then it follows that $(d(T_\infty y_n,T_\infty x))_{n \in \mathbb{N}}$ is a c-sequence. Then, by aplying the triangle inequality in the setting of the cone metric space over $\mathcal{A}$, we obtain that
$$ d(T_n x_n, T_\infty x) \preceq d(T_n x_n, T_\infty y_n) + d(T_\infty y_n, T_\infty x). $$
Since the right hand side from above is a c-sequence, then the left hand side, namely \\ $(d(T_n x_n,T_\infty x))_{n \in \mathbb{N}}$ is also a c-sequence and the proof is done.
\end{proof}

Now, it is time to show that under certain assumptions the (G)-limit of a sequence of mappings is unique. We have the following result.

\begin{theorem}\label{T3.5}
Let $(X,d)$ be a cone metric space over a given Banach algebra $\mathcal{A}$. Also, consider $X_n$ (for every $n \in \mathbb{N}$) and $X_\infty$ be some nonempty subsets of $X$. Suppose that the following assumptions are satisfied : 
\begin{align*}
& (i) \ \text{ for all } n \in \mathbb{N}, \text{ let } T_n \text{ to be a k-Lipschitz with respect to the Banach algebra } \mathcal{A}, \text{ i.e. there exists } k \in P, \\ 
& \text{ such that } d(T_n(x),T_n(y)) \preceq k \cdot d(x,y), \text{ for each } x,y \in X_n \\
& (ii) \ T_\infty : X_\infty \to X \text{ is a (G)-limit for the family } (T_n).
\end{align*}
Then, $T_\infty$ is the unique (G)-limit on $X_\infty$.
\end{theorem}

\begin{proof}
Let $T_\infty$ and $T_\infty^{\prime}$ be two (G)-limit mappings for the family $(T_n)$, defined on $X_\infty$. This means that for an arbitrary $x$ of $X_\infty$, there exists two sequences $(x_n)_{n \in \mathbb{N}}$ and $(y_n)_{n \in \mathbb{N}}$, with $x_n,y_n \in X_n$, such that :
\begin{align*}
& (d(x_n,x))_{n \in \mathbb{N}} \text{ and } (d(T_n x_n,T_\infty x))_{n \in \mathbb{N}} \text{ are c-sequences}, \\
& (d(y_n,x))_{n \in \mathbb{N}} \text{ and } (d(T_n y_n,T_\infty^{\prime} x))_{n \in \mathbb{N}} \text{ are also c-sequences}.
\end{align*}
Then, we obtain that
\begin{equation*}
\begin{split}
d(T_n x_n, T_n y_n) & \preceq k \cdot d(x_n,y_n) \\    
& \preceq k \left[ d(x_n,x) + d(y_n,x) \right] \\
& = k d(x_n,x) + k d(y_n,x).
\end{split}
\end{equation*}
Furthermore, since $(d(x_n,x))_{n \in \mathbb{N}}$ and $(d(y_n,x))_{n \in \mathbb{N}}$ are c-sequences, then it follows easily that $(k \cdot d(x_n,x))_{n \in \mathbb{N}}$ and $(k \cdot d(y_n,x))_{n \in \mathbb{N}}$ are also c-sequences. Then, by (3) of Lemma \ref{L1.12}, we get that $(k d(x_n,x)+k d(y_n,x))_{n \in \mathbb{N}}$ is also a c-sequence. This means that
\begin{align*}
& \text{for each arbitrary elements } c_1 \gg \theta, \ c_1 \in \mathcal{A}, \text{ there exists } N_1 \text{ that depends on } c_1, \\
& \text{ such that for every } n \geq N_1, \text{ one has that } k d(x_n,x) + k d(y_n,x) \ll c_1.      
\end{align*}
Then, for a fixed element $x \in X_\infty$, it follows that 
$$ d(T_\infty x, T_\infty^{\prime} x) \preceq d(T_\infty x, T_n x_n) + d(T_n x_n, T_n y_n) + d(T_n y_n, T_\infty^{\prime} x).$$
Since $(d(T_n x_n,T_\infty x)+d(T_n y_n,T_\infty^{\prime} x))_{n \in \mathbb{N}}$ is a c-sequence, it implies that for every $c_1 \gg \theta$, $c_1 \in \mathcal{A}$, there exists an index $N_1 = N_1(c_1) \in \mathbb{N}$, such that for every $n \geq N_1$, we have that $d(T_n x_n, T_n y_n) \preceq k d(x_n,x) + k d(y_n,x) \ll c_1$. Now, this implies that $(d(T_n x_n, T_n y_n))_{n \in \mathbb{N}}$ is a c-sequence. This leads to :
\begin{equation*}
\begin{split}
& (d(T_\infty x, T_n x_n) + d(T_n x_n, T_n y_n) + d(T_n y_n, T_\infty^{\prime} x))_{n \in \mathbb{N}} \text{ is a c-sequence, i.e. } \\
& \text{ for every } c \gg \theta,\ c \in \mathcal{A}, \text{ there exists an index } N = N(c) \in \mathbb{N}, \text{ such that for every } n \geq N, \\ 
& \text{ we have that } d(T_\infty x, T_\infty^{\prime} x) \preceq d(T_\infty x, T_n x_n) + d(T_n x_n, T_n y_n) + d(T_n y_n, T_\infty^{\prime} x) \ll c.
\end{split}    
\end{equation*}
Since for $c \gg \theta$, one has $0 \preceq d(T_\infty x, T_\infty^{\prime} x) \ll c$, following \cite{XuRadenovic} and (2) of Lemma \ref{L1.12}, we obtain that $d(T_\infty x, T_\infty^{\prime} x) = 0$, so the proof is over.
\end{proof}

Our third result from this section concerns the convergence of a sequence of fixed points of a family of mappings that has property (G), with respect to a given Banach algebra.

\begin{theorem}\label{T3.6}
Let $(X,d)$ be a cone metric space over a given Banach algebra $\mathcal{A}$. Also, consider $X_n$ (for $n \in \mathbb{N}$) and $X_\infty$ to be some nonempty subsets of $X$. Also, consider some mappings $T_n : X_n \to X$ and $T_\infty : X_\infty \to X$ that satisfy the following assumptions : 
\begin{align*}
& (i) \ \text{ for each } n \in \mathbb{N}, \ T_n \text{ is a k-contraction, i.e. there exists } k \in P \text{ with } \rho(k)<1, \\
& \text{ such that }  d(T_n(x),T_n(y)) \preceq k \cdot d(x,y), \text{ for each } x,y \in X_n, \\
& (ii) \ \text{ the family } (T_n) \text{ has property (G),} \\
& (iii) \ \text{ there exists } x_\infty \in F_{T_\infty}, \text{ i.e. } x_\infty \text{ is a fixed point of } T_\infty. 
\end{align*}
Then, $(d(x_n,x_\infty))_{n \in \mathbb{N}}$ is a c-sequence.
\end{theorem}

\begin{proof}
We know that $x_n = T_n(x_n)$ and that $x_\infty = T_\infty(x_\infty)$. Furthermore, since $T_\infty$ is a (G)-limit for the family $(T_n)$, then for an arbitrary element $x \in X_\infty$, there exists a sequence $(y_n)_{n \in \mathbb{N}}$, with $y_n \in X_n$ for each $n \in \mathbb{N}$, such that $(d(y_n,x))_{n \in \mathbb{N}}$ and $(d(T_n y_n, T_\infty x))_{n \in \mathbb{N}}$ are c-sequences. Moreover, taking $x = x_\infty$ we obtain that $(d(y_n,x_\infty))_{n \in \mathbb{N}}$ and $(d(T_n y_n, T_\infty x_\infty))_{n \in \mathbb{N}}$ are also c-sequences. Then, it follows that
\begin{equation*}
\begin{split}
d(x_n,x_\infty) & = d(T_n x_n, T_\infty x_\infty) \\
& \preceq d(T_n x_n, T_n y_n) + d(T_n y_n, T_\infty x_\infty) \\
& \preceq k \cdot d(x_n,y_n) + d(T_n y_n, T_\infty x_\infty) \\
& \preceq k \cdot d(x_n,x_\infty) + k \cdot d(y_n,x_\infty) + d(T_n y_n, T_\infty x_\infty). 
\end{split}    
\end{equation*}
This means that
\begin{equation*}
\begin{split}
& k \cdot d(x_n,x_\infty) + k \cdot d(y_n,x_\infty) + d(T_n y_n, T_\infty x_\infty) - d(x_n,x_\infty) \in P \Leftrightarrow \\
& k \cdot d(y_n,x_\infty) + d(T_n y_n, T_\infty x_\infty) + (k-e) \cdot d(x_n,x_\infty) \in P \Leftrightarrow \\
& k \cdot d(y_n,x_\infty) + d(T_n y_n, T_\infty x_\infty) - (e-k) \cdot d(x_n,x_\infty) \in P. 
\end{split}    
\end{equation*}
Now, since $\rho(k) < 1$ and $k \succeq \theta$, then there exists $(e-k)^{-1} \succeq \theta$. Furthermore $(e-k)^{-1} = \sum\limits_{i=0}^{\infty} k^i \in P$, since $k \in P$. At the same time, using the fact that $P^2 \subset P$ and multiplying by $(e-k)^{-1}$ , we obtain that $(e-k)^{-1} \left[ k d(y_n,x_\infty) + d(T_n y_n,T_\infty x_\infty) \right] - d(x_n,x_\infty) \in P$. This is equivalent to
$$ d(x_n,x_\infty) \preceq (e-k)^{-1} \left[ k d(y_n,x_\infty) + d(T_n y_n,T_\infty x_\infty) \right]. $$
Finally, since $(d(y_n,x_\infty))_{n \in \mathbb{N}}$ and $(d(T_n y_n,T_\infty x_\infty))_{n \in \mathbb{N}}$ are c-sequences, then also \\
$((e-k)^{-1}d(y_n,x_\infty))_{n \in \mathbb{N}}$ and $((e-k)^{-1}d(T_n y_n,T_\infty x_\infty))_{n \in \mathbb{N}}$ are c-sequences. So, for an arbitrary element $c \gg \theta$, $c \in \mathcal{A}$, there exists $N=N(c) \geq 0$, such that for every $n \geq N$, one has that $d(x_n,x_\infty) \preceq (e-k)^{-1} k d(y_n,x_\infty) + (e-k)^{-1} d(T_n y_n,T_\infty x_\infty) \ll c$, so the sequence $(d(x_n,x_\infty))_{n \in \mathbb{N}}$ is indeed a c-sequence.
\end{proof}

\begin{remark}\label{R3.7}
In Theorem \ref{T3.6} we supposed that indeed there exists $x_n \in F_{T_n}$. An alternative way is to suppose that $(X,d)$ is a complete cone metric space over $\mathcal{A}$ and after that one can establish a local variant of existence and uniqueness of fixed points for the mappings $T_n$, since they are contractions with respect to the cone metric, but not on the whole metric space.
\end{remark}

Now, it is time to present a consequence of Theorem \ref{T3.6} in which we refer to the connection between the pointwise convergence of a sequence of self-mappings and the (G)-property of the same sequence.
\begin{corollary}\label{C3.8}
Let $(X,d)$ be a cone metric space over a Banach algebra $\mathcal{A}$. Also, consider $T_n, T_\infty : X \to X$ some given mappings. Suppose the following assumptions are satisfied :
\begin{align*}
& (i) \ T_n \xrightarrow{p} T_\infty \text{ as } n \to \infty, \\
& (ii) \ T_n \text{ is a k-contraction with respect to the cone metric, for eac } n \in \mathbb{N}, \\
& (iii) \ \text{ there exists } x_n \in F_{T_{n}} \text{ and } x_\infty \in F_{T_{\infty}}. 
\end{align*}
Then, $(T_n)_{n \in \mathbb{N}}$ has the property $(G)$, with $T_{\infty}$ as the $(G)$-limit. 
\end{corollary}

\begin{proof}
From (i), it follows that for each $c \gg \theta$, $c \in \mathcal{A}$ and for every $x \in (X,d)$, there exists and index $N$ that depends on $c$ and $x$, such that for all $n \geq N$, one has that $d(T_n x,T_\infty x) \ll c$. We shall show that if $T_n \xrightarrow{p} T_\infty$, then the family $(T_n)$ has the property (G), with $T_\infty$ as the (G)-limit. For the case when $(T_n)$ has the (G) property, then for every $x \in X_\infty = X$, there exists $(x_{n})_{n \in \mathbb{N}}$, with $x_n \in X$ for each $n \in \mathbb{N}$, such that $(d(x_n,x))_{n \in \mathbb{N}}$ and $(d(T_n x_n, T_\infty x))_{n \in \mathbb{N}}$ are c-sequences. Furthermore, let's consider an arbitrary element $x \in X$. Taking $x_n = x$, for each $n \in \mathbb{N}$, we obtain that $d(x_n,x) = \theta \ll c$, so $(d(x_n,x))_{n \in \mathbb{N}}$ is a c-sequence. Moreover, $(d(T_n x_n, T_\infty x))_{n \in \mathbb{N}}$ is also a c-sequence, because of (i).
\end{proof}

Now, we shall present a theorem in which we are concerned with the relationship between the pointwise convergence of a sequence of mappings in the setting of cone metric spaces and the equicontinuity of the family of mappings.

\begin{theorem}\label{T3.9}
Let $(X,d)$ be a cone metric space over a Banach algebra $\mathcal{A}$ and $M$ be a nonempty subset of $X$. Furthermore, let $T_n : M \to X$ be a given operator such that the family $(T_n)$ has the (G) property with the (G)-limit $T_\infty$. Also, let's suppose that the following conditions are satisfied : 
\begin{align*}
& (i) \ \text{ the family } (T_n) \text{ is equicontinuous on } M, \\
& (ii) \text{ there exists } x_{n} \in F_{T_{n}}, \text{ for each } n \in \mathbb{N} \text{ and } x_\infty \in F_{T_{\infty}}.
\end{align*}
Then $T_{n} \xrightarrow{p} T_{\infty}$.
\end{theorem}

\begin{proof}
By (i), since the family $(T_n)$ is equicontinuous, it follows that for each $c_1 \gg \theta$, $c_1 \in \mathcal{A}$ and for every $x \in (X,d)$, there exists $c_2 \gg \theta$, $c_2 \in \mathcal{A}$ that depends on $c_1$ and $x$, such that for all $y \in (X,d)$ with $d(x,y) \ll c_2$, one has that $d(T_n x,T_n y) \ll c_1$. Let's suppose that $(T_n)$ has the (G) property with the (G)-limit $T_\infty$, i.e. for each $x \in M$, there exists a sequence $(x_n)_{n \in \mathbb{N}}$ from $M$, for which one has $(d(x_n,x))_{n \in \mathbb{N}}$ and $(d(T_n x_n,T_\infty x))_{n \in \mathbb{N}}$ are c-sequences. Moreover, we want to show that for every arbitrary element $c \gg \theta$, $c \in \mathcal{A}$ and for each $x \in (M,d)$, there exists an index $N \geq 0$ that depends on $c$ and $x$ such that for every $n \geq N$, one has $d(T_n x,T_\infty x) \ll c$. So, let $c \gg \theta$ be a fixed arbitrary element of the given Banach algebra and $x \in M \subset X$.  From the equicontinuity of the family $(T_n)$ over $\mathcal{A}$, there exists $\bar{c}$ that depends on $c$ and $x$, where $\bar{c}$ is from $P$, such that $d(x,y) \ll \bar{c}$ implies that $d(T_n x, T_n y) \ll c$, with $y \in M$ and $n \in \mathbb{N}$. For $x$ and $c$, there exists an index $N_1 \geq 0$ that depends on $c$ and $x$, such that for every $n \geq N_1$, one has that $d(T_n x_n, T_n x) \ll c$. Taking $n \geq \max \lbrace N,N_1 \rbrace$, we have that 
$$ d(T_n x, T_\infty x) \preceq d(T_n x_n, T_\infty x) + d(T_n x_n, T_n x) \ll c,$$
where we have used the fact that $(d(T_n x_n,T_\infty x))_{n \in \mathbb{N}}$ and $(d(T_n x_n,T_n x))_{n \in \mathbb{N}}$ are c-sequences, so their sum is also a c-sequence by (3) of Lemma \ref{L1.12}. Also, $N_2$ is the index that is found out from the fact that $(d(T_n x_n,T_\infty x))_{n \in \mathbb{N}}$ is a c-sequence. Finally, we recall that we also have used the idea that if $a,b$ and $c^{\prime}$ are elements from $\mathcal{A}$, such that $a \preceq b$ and $b \ll c^{\prime}$, then $a \ll c^{\prime}$.
\end{proof}

Now, the next theorem of this section is an existence result for the fixed points of the (G)-limit mapping of a sequence of contractions with respect to the cone metric space over a given Banach algebra.

\begin{theorem}\label{T3.10}
Let $(X,d)$ be a cone metric space over a Banach algebra $\mathcal{A}$. Also, consider $X_n$ and $X_\infty$ some given nonempty subsets of $X$. Let $T_n : X_n  \to X$ and $T_\infty : X_\infty \to X$ be some mappings that satisfy :
\begin{align*}
& (i) \ \text{ the family } (T_n) \text{ has property (G) with the (G)-limit } T_\infty, \\
& (ii) \ T_n \text{ are k-contractions in the sense of the given cone metric}, \\
& (iii) \ \text{ there exists } x_n \in F_{T_{n}}.
\end{align*}
Then, there exists $x_\infty \in F_{T_\infty}$ if and only if the sequence $(x_n)_{n \in \mathbb{N}}$ is convergent in $X_\infty$ in the sense of the Banach algebra (i.e. there exists $y \in X_\infty$ such that $(d(x_n,y))_{n \in \mathbb{N}}$ is a c-sequence).
\end{theorem}

\begin{proof}
From Theorem \ref{T3.6} it follows that if there exists $x_\infty \in F_{T_{\infty}}$, then $(d(x_n,x_\infty))_{n \in \mathbb{N}}$ is a c-sequence. Now, we consider the reverse implication, namely let $(x_{n})_{n \in \mathbb{N}}$, with $x_{n} \in X_{n}$ for each $n \in \mathbb{N}$ and $x_\infty \in X_\infty$ be such that $(d(x_n,x_\infty))_{n \in \mathbb{N}}$ is a c-sequence. For the element $x_\infty$, by (i) we have that there exists a sequence $(y_n)_{n \in \mathbb{N}}$, with $y_n \in X_n$ for every $n \in \mathbb{N}$ such that $(d(y_n,x_\infty))_{n \in \mathbb{N}}$ and $(d(T_n y_n,T_\infty x_\infty))_{n \in \mathbb{N}}$ are c-sequences. Then, we obtain 
\begin{equation*}
\begin{split}
d(x_\infty,T_\infty x_\infty) & \preceq d(x_\infty,x_n) + d(T_n x_n, T_n y_n) + d(T_n y_n,T_\infty x_\infty)  \\
& \preceq d(x_\infty,x_n) + k \cdot d(x_n,y_n) + d(T_n y_n,T_\infty x_\infty) \\
& \preceq d(x_\infty,x_n) + k \cdot d(x_n,x_\infty) + k \cdot d(y_n,x_\infty) + d(T_n y_n,T_\infty x_\infty).
\end{split}    
\end{equation*}
Since all the elements from the right hand side are c-sequences, by (3) of Lemma \ref{L1.12}, the whole sum from the right hand side is a c-sequence. This means that for an arbitrary $c \gg \theta$ we have that $ 0 \preceq d(x_\infty, T_\infty x_\infty) \ll c$. By (2) of Lemma \ref{L1.12} it follows that $x_\infty = T_\infty x_\infty$, so the proof is over.
\end{proof}

Now, we are ready to present our last two results from the present section regarding the link between the uniform convergence with respect to the cone metric and the (H)-property of a given sequence of operators.

\begin{theorem}\label{T3.11}
Let $(X,d)$ be a cone metric space over a Banach algebra $\mathcal{A}$. Also, consider $M \subset X$ a nonempty set. Let $T_n,T_\infty : M \to X$ some given mappings. \\
a) If $T_n \xrightarrow{u} T_\infty$, then $T_\infty$ is the (H)-limit of the family $(T_n)$. \\
b) If $T_\infty$ is the (H)-limit of $(T_n)$ and if $T_\infty$ is uniformly continuous on $M$, then $T_n \xrightarrow{u} T_\infty$.
\end{theorem}

\begin{proof}
a) Suppose that $T_n \xrightarrow{u} T_\infty$, i.e. for each $c \gg \theta$, with $c \in \mathcal{A}$, there exists an index $N = N(c) \geq 0$, such that for all $n \geq N$, it follows that $d(T_n x,T_\infty x) \ll c$, for each arbitrary $x$. Let's consider a sequence $(x_n)_{n \in \mathbb{N}}$, such that $x_n \in X_n$ for every $n \in \mathbb{N}$. Taking $y_n = x_n \in X_n = X_\infty = M$, we only need to show that $(d(T_n x_n,T_\infty y_n))_{n \in \mathbb{N}} = (d(T_n x_n,T_\infty x_n))_{n \in \mathbb{N}}$ is a c-sequence. Finally, taking $x = x_n$ in the definition of uniform convergence with respect to the cone metric $d$ of the family $(T_n)$, then the proof is over. \\
b) We let the proof to the reader, since it follows in a similar way \cite{BarbetNachi}, namely the one from the case of metric spaces. Furthermore, the concept of uniform continuity in the framework of a cone metric space over a Banach algebra can be extended from the case of usual metric spaces.
\end{proof}

Now, our last theorem of this section concerns the convergence of a sequence of fixed points of a family of mappings to the fixed point of the (H)-limit of the same family of operators.

\begin{theorem}\label{T3.12}
Let $(X,d)$ be a cone metric space over a Banach algebra $\mathcal{A}$. Consider $X_n$ for each $n \in \mathbb{N}$ and $X_\infty$ to be some nonempty subset of $X$. Also, let $T_n : X_n \to X$ and $T_\infty : X_\infty \to X$ be some mappings that satisfy the following assumptions :
\begin{align*}
& (i) \ x_n \in F_{T_n}, \\
& (ii) \ (T_n) \text{ has the property (H) with the (H)-limit } T_\infty, \\
& (iii) \ T_\infty \text{ is a } k_\infty-\text{contraction with respect to the cone metric } d, \\
& (iv) \ \text{ there exists and is unique } x_\infty \in F_{T_{\infty}}.
\end{align*}
Then $(d(x_n,x_\infty))_{n \in \mathbb{N}}$ is a c-sequence.
\end{theorem}

\begin{proof}
From the property (H) and for the sequence $(x_n)_{n \in \mathbb{N}}$, there exists another sequence $(y_n)_{n \in \mathbb{N}}$ from $X_\infty$, for which one has that $(d(x_n,y_n))_{n \in \mathbb{N}}$ and $(d(T_n x_n,T_\infty y_n))_{n \in \mathbb{N}}$ are c-sequences. Furthermore, we consider the following chain of inequalities :
\begin{equation*}
\begin{split}
d(x_n,x_\infty) & = d(T_n x_n,T_\infty x_\infty) \\
& \preceq d(T_n x_n,T_\infty y_n) + d(T_\infty y_n,T_\infty x_\infty) \\
& \preceq d(T_n x_n,T_\infty y_n) + k_\infty \cdot d(y_n,x_\infty) \\
& \preceq d(T_n x_n,T_\infty y_n) + k_\infty \cdot d(y_n,x_n) + k_\infty \cdot d(x_n,x_\infty).
\end{split}    
\end{equation*}
After some easy algebraic manipulations, one obtains that 
\begin{align*}
d(T_n x_n,T_\infty y_n) + k_\infty \cdot d(y_n,x_n) - (e-k_\infty) \cdot d(x_n,x_\infty) \in P.    
\end{align*}
As in the proofs of the above theorems, using the property of the solid cone $P$, namely $P^2 \subset P$, the fact that $\rho(k_{\infty}) < 1$ and multiplying by $(e-k_\infty)^{-1}$, it follows that 
\begin{align*}
(e-k)^{-1} \cdot d(T_n x_n,T_\infty y_n) + (e-k)^{-1} \cdot k_\infty \cdot d(y_n,x_n) - d(x_n,x_\infty) \in P.      
\end{align*}
This is equivalent to 
\begin{align*}
d(x_n,x_\infty) \preceq (e-k)^{-1} \cdot d(T_n x_n,T_\infty y_n) + (e-k)^{-1} \cdot k_\infty \cdot d(y_n,x_n).    
\end{align*}
Using the fact that, by Lemma \ref{L1.12}, the right hand side is a c-sequence, the conclusion follows easily.
\end{proof}

At last, we are ready to give a crucial remark regarding the assumption (iv) from Theorem \ref{T3.12}.  

\begin{remark}\label{R3.13}
One can omit condition (iv) from the previous theorem if we suppose that $(X,d)$ is a complete cone over the Banach algebra $\mathcal{A}$. With this assumption, since $T_\infty$ is a contraction on a subset of the space in the sense of the cone metric $d$, then one can prove a local variant principle in which the operator has a unique fixed point. 
\end{remark}

\section{Applications to systems of functional and differential equations}

In this section we shall present some applications linked to functional coupled equations and systems of differential equations, respectively. Also, we shall show that our theorems from the second section are a viable tool for studying the convergence of the unique solution of different types of sequences regarding generalized type of functional and differential equations. Furthermore, in our first result, following Theorem 3.1 of \cite{HuangRadenovic} we shall present the convergence of the solutions of some coupled equations, using our results that are based upon the idea of an Banach algebra.

\begin{theorem}\label{T4.1}
Let $F_n, G_n, \tilde{F}, \tilde{G} : \mathbb{R}^2 \to \mathbb{R}^2$ be some given mappings (for $n \in \mathbb{N}$). Also, consider the following systems of coupled functional equations :
\begin{equation}\label{EQ4.1}
\begin{cases}
F_n (x,y) = 0 \\
G_n (x,y) = 0
\end{cases}
, \ \text{ with } (x,y) \in \mathbb{R}^2 \ ,
\end{equation}
and 
\begin{equation}\label{EQ4.2}
\begin{cases}
\tilde{F}(x,y) = 0 \\
\tilde{G}(x,y) = 0
\end{cases}
, \ \text{ with } (x,y) \in \mathbb{R}^2.
\end{equation}
Suppose that the mappings $F_n$, $G_n$, $\tilde{F}$ and $\tilde{G}$ satisfy the following assumptions : \\
(1) There exists $M > 0$, such as for $n \in \mathbb{N}$, there exists $L_n > 0$ satisfying $\max\limits_{n \geq 1} L_n \leq M < 1$, such that
\begin{equation*}
\begin{split}
\begin{cases}
& | F_n(x_1,y_1) - F_n(x_2,y_2) + x_1 - x_2 | \leq L_n |x_1 - x_2| \\
& | G_n(x_1,y_1) - G_n(x_2,y_2) + y_1 - y_2 | \leq L_n |y_1 - y_2| \\
\end{cases}
\end{split}    
\end{equation*}
where $(x_1,x_2)$ and $(y_1,y_2)$ are from $\mathbb{R}^2$. \\
(2) There exists $\tilde{L} \in (0,1)$, such that
\begin{equation*}
\begin{split}
\begin{cases}
& | \tilde{F}(x_1,y_1) - \tilde{F}(x_2,y_2) + x_1 - x_2 | \leq \tilde{L} |x_1 - x_2| \\
& | \tilde{G}(x_1,y_1) - \tilde{G}(x_2,y_2) + y_1 - y_2 | \leq \tilde{L} |y_1 - y_2| \\
\end{cases}
\end{split}    
\end{equation*}
where $(x_1,x_2)$ and $(y_1,y_2)$ are from $\mathbb{R}^2$. \\
(3) The sequence $(F_n)_{n \in \mathbb{N}}$ converges pointwise to $\tilde{F}$ and $(G_n)_{n \in \mathbb{N}}$ also converges pointwise to $\tilde{G}$ in the classical sense, i.e. :
\begin{equation*}
\begin{split}
\begin{cases}
& F_n \xrightarrow{p} \tilde{F} \\
& G_n \xrightarrow{p} \tilde{G}
\end{cases}
, \text{ i.e. } \ \ 
\begin{cases}
& \lim\limits_{n \to \infty} F_n(x) = \tilde{F}(x) \\
& \lim\limits_{n \to \infty} G_n(x) = \tilde{G}(x)
\end{cases}
, \ \text{ for each } x \in \mathbb{R}^2.
\end{split}    
\end{equation*}
Then $x_n$ converges to $\tilde{x}$ and $y_n$ converges to $\tilde{y}$, where $(x_n,y_n)$ is the unique solution of \ref{EQ4.1} and $(\tilde{x},\tilde{y})$ is the unique solution of \ref{EQ4.2}.
\end{theorem}

\begin{proof}
Let's consider the Banach algebra $\mathcal{A} = \mathbb{R}^2$ from Example \ref{E2.7}. Also, let $X = \mathbb{R}^2$. Furthermore, consider the operators $T_n,T : X \to X$, defined as
\begin{align*}
& T_n(x,y) = (F_n(x,y)+x,G_n(x,y)+y) \ \text{ and } \\
& T(x,y) = (\tilde{F}(x,y)+x,\tilde{G}(x,y)+y) , \ 
\end{align*}
for each $(x,y) \in \mathbb{R}^2$. As in \cite{HuangRadenovic} it is easy to see that for every $n \in \mathbb{N}$, the operator $T_n$ is a contraction with coefficient $(L_n,0)$ and that $T$ is also a contraction with respect to the Banach algebra $\mathcal{A}$, but with coefficient $(\tilde{L},0)$. Moreover, also following \cite{HuangRadenovic}, one can easily show that for each $n \in \mathbb{N}$, there exists and is unique $(x_n,y_n)$ the solution for the coupled equation system \ref{EQ4.1} and $(\tilde{x},\tilde{y})$ the solution for \ref{EQ4.2}, respectively. Now, our aim is to show that $(d((x_n,y_n),(\tilde{x},\tilde{y})))_{n \in \mathbb{N}}$ is a c-sequence with respect to the Banach algebra $\mathcal{R}^2$ endowed with the cone metric $d$ from Example \ref{E2.7}. Now, in Theorem \ref{T2.10} we considered contractions with the same coefficient $\alpha$, with $\rho(\alpha) < 1$. At the same time, one can easily see that the same theorem can be proved in the case when the operators have different contraction coefficients $\alpha_n$ for $n \in \mathbb{N}$, but endowed with the property that the sequence $(\alpha_n)_{n \in \mathbb{N}}$ is bounded, i.e. there exists $M \in P$ satisfying $\rho(M) < 1$, such as $\alpha_n \preceq M$, for all $n \in \mathbb{N}$. In our case, one can see that $\rho((L_n,0)) = \lim\limits_{n \to \infty} \| (L_n,0)^{n \|^{1/n}}$ and that $(L_n,0) \preceq (M,0)$, with $\rho(M,0) = M < 1$. Finally, the property that the sequence of coefficients is bounded leads to $\max\limits_{n \in \mathbb{N}} L_n \leq M < 1$. \\
For example, we can modify the proof of Theorem \ref{T2.10}, with : 
$$ d(x_{n}^{\ast}, x^{\ast}) \preceq \alpha_n d(x_{n}^{\ast},x^{\ast}) + d(T_n x^{\ast}, T x^{\ast}) \preceq M d(x_{n}^{\ast},x^{\ast}) + d(T_n x^{\ast}, T x^{\ast}),$$
when $\alpha_n \preceq M$, for every $n \in \mathbb{N}$. Here we have used the fact that if $\alpha_n \preceq M$, then $\alpha_n d_n \preceq M d_n$, where $d_n \in P$, because this leads to $d_n \cdot (M - \alpha_n) \in P$. This is of course a valid affirmation, because we can use the property $P^2 \subseteq P$ on $d_n$ and $M-\alpha_n$, which are both from the solid cone $P$. Moreover, we see that it is also a valid assumption that $(e-M)^{-1} \in P$, since $M$ is from the nonempty cone $P$.  \\
Now, using the above reasoning with respect to Theorem \ref{T2.10}, we only need to show that $T_n \xrightarrow{p} T$, i.e. :
\begin{align*}
& \text{ for each } c \gg (0,0) \text{ and for } z \in \mathbb{R}^2, \text{ there exists } N = N(c,z) \geq 0, \text{ such that for every } n \geq N, \\
& \text{ we must have that } d(T_n z,T z) \ll c.    
\end{align*}
Now, let $c = c(c_1,c_2) \in \mathbb{R}^2$, with $c_1,c_2 > 0$ and let $z=(x,y)$ an arbitrary element from $\mathbb{R}^2$. Then, we only need to show that
\begin{equation*}
\begin{split}
& (c_1,c_2) - d((F_n(x,y)+x,G_n(x,y)+y),(\tilde{F}(x,y)+x,\tilde{G}(x,y)+y)) \in int(P) \Leftrightarrow \\
& (c_1,c_2) - (|F_n(x,y)-\tilde{F}(x,y)|,|G_n(x,y)-\tilde{G}(x,y)|) \in int(P).
\end{split}    
\end{equation*}
This is equivalent to showing that
\begin{equation*}
\begin{cases}
& |F_n(x,y)-\tilde{F}(x,y)| < c_1 \ , \\
& |G_n(x,y)-\tilde{G}(x,y)| < c_2 \ .
\end{cases}    
\end{equation*}
Now, from the assumption $(3)$, we obtain that for $c_1 > 0$ and for $z=(x,y) \in \mathbb{R}^2$, there exists $N_1 = N_1(c_1,x,y) \geq 0$, such as for every $n \geq N_1$, it follows that $|F_n(x,y)-\tilde{F}(x,y)| < c_1$. In a similar way, for $c_2 > 0$ and for $z=(x,y) \in \mathbb{R}^2$, there exists $N_2 = N_1(c_2,x,y) \geq 0$, such as for every $n \geq N_2$, it follows that $|G_n(x,y)-\tilde{G}(x,y)| < c_2$. Finally, taking $n \geq N:= \max \lbrace N_1,N_2 \rbrace$, the conclusion follows easily.
\end{proof}

Now, our second crucial result from the present section concerns an application to systems of differential equations. In fact, using the results from the second section and the idea of a Banach algebra, we present an existence and uniqueness theorem for the solution of a nonlinear systems of differential equations.

\begin{theorem}\label{T4.2}
Let $D \subset \mathbb{R}^3$ and $(\alpha,\beta,\gamma) \in D$. Also, consider $\bar{\alpha}$, $\bar{\beta}$ and $\bar{\gamma}$ sufficiently small such that the compact set $\Delta := \lbrace (x,y,z) \ / \ |x-a| \leq \bar{a}, \ |y-\beta| \leq \bar{\beta}, \ |z-\gamma| \leq \bar{\gamma} \rbrace$ is a subset of $D$. \\
Consider the following nonlinear system of differential equations :
\begin{equation}\label{EQ4.3}
\begin{cases}
& y^\prime(x) = f(x,y(x),z(x)) \\
& z^\prime(x) = g(x,y(x),z(x)) \\
& y(a) = \beta \\
& z(a) = \gamma
\end{cases}
, \text{ where } x \in I := [a-\bar{a},a+\bar{a}].
\end{equation}
Also, suppose the following assumptions are satisfied :
\begin{equation*}
\begin{split}
& (1) \ \text{ the mappings } f \text{ and } g \text{ are continuous on } D, \\
& (2) \ \text{ there exists } L_1,L_2 > 0, \text{ such that } 
\end{split}
\end{equation*}
\begin{equation*}
\begin{split}
& \begin{cases}
& |f(x,y,z)-f(x,\bar{y},\bar{z})| \leq L_1 |y-\bar{y}| \\
& |g(x,y,z)-g(x,\bar{y},\bar{z})| \leq L_2 |z-\bar{z}|
\end{cases}
\ , \text{ for every } (x,y,z) \text{ and } (x,\bar{y},\bar{z}) \in D.
\end{split}
\end{equation*}
Then, there exists a unique solution for the nonlinear differential system \ref{EQ4.3} on $I = [a-\bar{a},a+\bar{a}]$. 
\end{theorem}

\begin{proof}
First of all, we observe that the system of differential equation \ref{EQ4.3} can be written under the following integral form :
\begin{equation}\label{EQ4.4}
\begin{split}
\begin{cases}
& y(x) = \beta + \int\limits_{a}^{x} f(s,y(s),z(s)) \ ds \\
& z(x) = \gamma + \int\limits_{a}^{x} g(s,y(s),z(s)) \ ds
\end{cases}
\ , \text{ where } x \in I.
\end{split}    
\end{equation}
Furthermore, we consider the operator $T : C(I) \times C(I) \to C(I) \times C(I)$, defined as 
\begin{align*}
T(y,z)(x) = \left( \beta + \int\limits_{a}^{x} f(s,y(s),z(s)) \ ds, \gamma + \int\limits_{a}^{x} g(s,y(s),z(s)) \ ds \right), \text{ with } x \in I.    
\end{align*}
In a more simplified form, $T$ can be written as $T=(T_1,T_2)$, with $T(y,z) = \left( T_1(y,z), T_2(y,z) \right)$. Moreover, for each $x \in I$, one has that
$$ T(y,z)(x) = \left( T_1(y,z)(x), T_2(y,z)(x) \right). $$
Here $T_1,T_2 : C(I) \times C(I) \to C(I)$, where
\begin{equation}\label{EQ4.5}
\begin{split}
\begin{cases}
& T_1(y,z)(x) = \beta + \int\limits_{a}^{x} f(s,y(s),z(s)) \ ds \\
& T_2(y,z)(x) = \gamma + \int\limits_{a}^{x} g(s,y(s),z(s)) \ ds
\end{cases}
\ , \text{ where } x \in I.
\end{split}    
\end{equation}
Now, it is time to recall that the space $C(I)$ can be endowed with two norms, namely for all $x \in C(I)$, one has the Chebyshev norm $\| x \|_{C} = \max\limits_{t \in I} |x(t)|$ and the Bielecki norm $\| x \|_{B,\tau} = \max\limits_{t \in I} |x(t)| e^{-\tau (t-(a-\bar{a}))}$. Now, since we shall work very often with $C(I)$, it is intuitive to specify what norm is appropriate in each particular case. So, let's denote $\bar{X} = C(I,\| \cdot \|_{B,\tau_1}) \times C(I,\| \cdot \|_{B,\tau_2})$. One the other hand, we define the operator $d : \bar{X} \times \bar{X} \to \mathbb{R}^2$, by 
\begin{align*}
d(a_1,a_2) = d((y_1,z_1),(y_2,z_2)) = \left( \| y_1 - y_2 \|_{B,\tau_1}, \| z_1 - z_2 \|_{B,\tau_2} \right) \ ,    
\end{align*}
where $a_1 = (y_1,z_1)$ and $a_2 = (y_2,z_2)$ are from $\bar{X}$. This means that $y_1,y_2 \in C(I,\| \cdot \|_{B,\tau_1})$ and $z_1,z_2 \in C(I,\| \cdot \|_{B,\tau_2})$, respectively. Now, it is time to show that the mapping $d$ is a complete cone metric over $\mathbb{R}^2$. From Example \ref{E2.7}, we know that $\mathbb{R}^2$ is a Banach algebra with the solid 'positive' cone $P = \lbrace (y,z) \ / \ y \geq 0 \text{ and } z \geq 0 \rbrace \subset \mathbb{R}^2$. First of all, we shall show that $d$ satisfies the basic axioms of the cone metric over $\mathbb{R}^2$ : \\
(I) For each $(a_1,a_2) = ((y_1,z_1),(y_2,z_2)) \in \bar{X} \times \bar{X}$, we have that $d(a_1,a_2) \succeq (0,0)$ is equivalent to $\| y_1 - y_2 \|_{B,\tau_1} \geq 0$ and $\| z_1 - z_2 \|_{B,\tau_2} \geq 0$, which is evidently true. \\
(II) For each $(a_1,a_2) = ((y_1,z_1),(y_2,z_2)) \in \bar{X} \times \bar{X}$, we have that $d(a_1,a_2) = d((y_1,z_1),(y_2,z_2)) = \left( \| y_1 - y_2 \|_{B,\tau_1}, \| z_1 - z_2 \|_{B,\tau_2} \right)$ and that $d(a_2,a_1) = d((y_2,z_2),(y_1,z_1)) = \left( \| y_2 - y_1 \|_{B,\tau_1}, \| z_2 - z_1 \|_{B,\tau_2} \right)$, so the second axiom is also satisfied. \\
(III) Now, for the triangle inequality, we consider $a_1=(y_1,z_1)$, $a_2=(y_2,z_2)$ and $a_3=(y_3,z_3)$ three arbitrary elements from $\bar{X}$. We must show that $d(a_1,a_3) \preceq d(a_1,a_2) + d(a_2,a_3)$, which is equivalent to $d(a_1,a_2) + d(a_2,a_3) - d(a_1,a_3) \in \ P$. This leads to
\begin{equation*}
\begin{split}
\begin{cases}
& \| y_2 - y_3 \|_{B,\tau_1}  + \| y_1 - y_2 \|_{B,\tau_1}  \geq \| y_1 - y_3 \|_{B,\tau_1}  \\
& \| z_2 - z_3 \|_{B,\tau_2}  + \| y_z - z_2 \|_{B,\tau_2}  \geq \| z_1 - z_3 \|_{B,\tau_2}  
\end{cases}
\ ,
\end{split}    
\end{equation*}
which is also valid. Now, regarding $d$ and $\bar{X}$, we must show that $(\bar{X},d)$ is complete with respect to the setting of the Banach algebra $\mathbb{R}^2$. For this, let $(x_n)_{n \in \mathbb{N}} \subset \bar{X}$ be a Cauchy sequence in the sense of the Banach algebra $\mathbb{R}^2$. We shall show that this sequence is convergent. Now, since $x_n$ can be written as $x_n=(y_n,z_n)$ for every $n \in \mathbb{N}$, with $y_n \in C(I, \| \cdot \|_{B,\tau_1})$ and $z_n \in C(I, \| \cdot \|_{B,\tau_2})$ and since the sequence $(x_n)_{n \in \mathbb{N}}$ is Cauchy, then $(d(x_n,x_m))_{n \in \mathbb{N}}$ is a c-sequence. Consider $c_1, c_2 > 0$ arbitrary elements of $\mathbb{R}$. Taking $c=(c_1,c_2)$, there exists $N \geq 0$ that depends on $c$, such that for every $n,m \geq N$, one has that $(c_1,c_2) - d(x_n,x_m) \in \ int(P)$. Now, this leads to $\| y_n - y_m \|_{B,\tau_1} < c_1$ and $\| z_n - z_m \|_{B,\tau_2} < c_2$, respectively. Also, since $c_1$ and $c_2$ are arbitrary elements of $\mathbb{R}$, then it follows that
\begin{align*}
& \lim\limits_{n,m \to \infty} \| y_n - y_m \|_{B,\tau_1} = 0 \ , \\
& \lim\limits_{n,m \to \infty} \| z_n - z_m \|_{B,\tau_2} = 0 \ , \\
\end{align*}
i.e. $(y_n)_{n \in \mathbb{N}}$ is Cauchy with respect to $X_1$ and $(z_n)_{n \in \mathbb{N}}$ is Cauchy with respect to $X_2$, where $X_1 := C(I,\| \cdot \|_{B,\tau_1})$ and $X_2 := C(I,\| \cdot \|_{B,\tau_2})$. Now, since $X_1$ and $X_2$ are both Banach spaces, it follows that $(y_n)_{n \in \mathbb{N}}$ is convergent in $X_1$ and $(z_n)_{n \in \mathbb{N}}$ is convergent in $X_2$, respectively. This means that there exists $\bar{y}_0 \in X_1$ and $\bar{z}_0 \in X_2$, for which one has that $y_n \to \bar{y}_0$ and $z_n \to \bar{z}_0$ as $n \to \infty$. By convergence, for $c_1, c_2 > 0$, there exists $N_1 = N(c_1) \geq 0$ and $N_2 = N_2(c_2) \geq 0$, such that for every $n \geq N_1$ and $n \geq N_2$ simultaneously, one has that $\| y_n - \bar{y}_0 \|_{B,\tau_1} < c_1$ and $\| z_n - \bar{z}_0 \|_{B,\tau_2} < c_2$. So, it follows that $(c_1,c_2) - d((y_n,z_n),(\bar{y}_0,\bar{z}_0)) \in \ int(P)$, which leads to our desired conclusion. \\
Now, since we have showed that $(\bar{X},d)$ is a complete cone metric space over the Banach algebra $\mathbb{R}^2$, then we have that $T : \bar{X} \to \bar{X}$, with $T = (T_1,T_2)$, where $T_1 : \bar{X} \to X_1$ and $T_2 : \bar{X} \to X_2$. \\
Now, let $\Delta := \lbrace (x,y,z) \ / \ |x-a| \leq \bar{a}, \ |y-\beta| \leq \bar{\beta}, \ |z-\gamma| \leq \bar{\gamma} \rbrace \subset D$. Since the mappings $f$ and $g$ are continuous on the compact $\Delta$, then there exists $M_f,M_g \geq 0$, such that $|f(x,y,z)| \leq M_f$ and $|g(x,y,z)| \leq M_g$, for each $(x,y,z) \in \Delta$. Furthermore, let
\begin{equation*}
\begin{cases}
& h_1 := \min \Big\lbrace \bar{a}, \dfrac{\bar{\beta}}{M_f} \Big\rbrace \\
& h_2 := \min \Big\lbrace \bar{a}, \dfrac{\bar{\gamma}}{M_g} \Big\rbrace.
\end{cases}
\end{equation*}
Also, define : 
\begin{align*}
& S := \lbrace (y,z) \in C([a-h_1,a+h_1], \| \cdot \|_{B,\tau_1}) \times C([a-h_2,a+h_2], \| \cdot \|_{B,\tau_2}) \ / \ \| y - \beta \|_{B,\tau_1} \leq \bar{\beta} \\
& \text{ and } \| z - \gamma \|_{B,\tau_2} \leq \bar{\gamma} \rbrace.
\end{align*}
Now, we shall show that $T : S \to S$, i.e. taking $w \in S$, we show that $Tw \in S$. For this, consider $w=(y,z) \in S$ and $Tw=(T_1w,T_2w)  \in S$. This means that we need to show
\begin{equation*}
\begin{split}
\begin{cases}
& T_1 w \in C([a-h_1,a+h_1], \| \cdot \|_{B,\tau_1}) \\
& T_2 w \in C([a-h_2,a+h_2], \| \cdot \|_{B,\tau_2})
\end{cases}
, \ \text{ with }
\begin{cases}
& \| T_1 w - \beta \|_{B,\tau_1} \leq \bar{\beta} \\
& \| T_2 w - \gamma \|_{B,\tau_2} \leq \bar{\gamma} 
\end{cases}
\end{split}    
\end{equation*}
The last chain of inequalities is equivalent to
\begin{equation*}
\begin{cases}
& \max_{x \in [a-h_1,a+h_1]} |T_1(y,z)(x) - \beta| e^{-\tau_1(x-(a-\bar{a}))} \leq \bar{\beta} \\
& \max_{x \in [a-h_2,a+h_2]} |T_2(y,z)(x) - \gamma| e^{-\tau_2(x-(a-\bar{a}))} \leq \bar{\gamma}
\end{cases}    
\end{equation*}
So, it follows that
\begin{align*}
& |T_1(y,z)(x)-\beta| = \Bigg| \beta + \int\limits_{a}^{x} f(s,y(s),z(s)) \ ds - \beta \Bigg| \leq \int\limits_{a}^{x} |f(s,y(s),z(s))| \ ds \leq M_f |x-a| \leq M_f h_1 \leq \bar{\beta}.  \\
& |T_2(y,z)(x)-\gamma| = \Bigg| \gamma + \int\limits_{a}^{x} g(s,y(s),z(s)) \ ds - \gamma \Bigg| \leq \int\limits_{a}^{x} |g(s,y(s),z(s))| \ ds \leq M_g |x-a| \leq M_g h_2 \leq \bar{\gamma}.   
\end{align*}
For $T=(T_1,T_2) : S \to S$, we shall show that $T$ is a contraction with respect to the cone metric $d$ over $\mathbb{R}^2$. First of all, we show that $S \subset \bar{X}$ is complete in the setting of $d$, i.e. $(S,d)$ is a complete cone metric space over the same Banach algebra as before. For this, let's consider the sequence $(x_n)_{n \in \mathbb{N}}$, with $x_n=(y_n,z_n)$ and $x_n \in S$ for each $n \in \mathbb{N}$, such that $(x_n)_{n \in \mathbb{N}}$ is Cauchy. So, we shall show that $(x_n)_{n \in \mathbb{N}}$ is convergent with respect to $d$. As we have done with the proof of completeness of $\bar{X}$, we find that $(y_n)_{n \in \mathbb{N}}$ is convergent in $S_1$ and $(z_n)_{n \in \mathbb{N}}$ is convergent in $S_2$, where :
\begin{align*}
& S_1 := \Bigg\lbrace y \in C([a-h_1,a+h_1], \| \cdot \|_{B,\tau_1}) \ / \ \| y - \beta \|_{B,\tau_1} \leq \bar{\beta} \Bigg\rbrace \ , \\
& S_2 := \Bigg\lbrace y \in C([a-h_2,a+h_2], \| \cdot \|_{B,\tau_2}) \ / \ \| z - \gamma \|_{B,\tau_2} \leq \bar{\gamma} \Bigg\rbrace \ , \\
\end{align*}
We have used the fact that $S_1$ is a closed subset of $X_1$ and $X_1$ is complete, then $S_1$ is also complete with respect to $\| \cdot \|_{B,\tau_1}$. In a similar way, since $S_2$ is closed and $X_2$ is complete, then $S_2$ is complete with respect to $\| \cdot \|_{B,\tau_2}$. Then, it is easy to see that $(S,d)$ is complete with respect to the Banach algebra $\mathbb{R}^2$. \\
On the other hand, for $T = (T_1,T_2) : S \to S$, we show that $T$ is a cone contraction, i.e. there exists $\alpha = (\alpha_1,\alpha_2)$, with $\alpha_1, \alpha_2 \geq 0$, such that $\rho(\alpha) < 1$ and for each $w,\bar{w} \in S$, one has that $d(Tw,T\bar{w}) \preceq \alpha \cdot d(w,\bar{w})$. For the simplicity of results, we can take $\alpha_2 = 0$, because we can work with the definition of contraction with respect to the cone metric with the assumption that $\alpha \succeq \theta$ and not $\alpha \gg \theta$. Then, it follows that
\begin{align*}
(\alpha_1,0) \cdot d((y_1,z_1),(y_2,z_2)) - d((T_1(y_1,z_1),T_2(y_1,z_1)),(T_1(y_2,z_2),T_2(y_2,z_2))) \succeq (0,0).    
\end{align*}
This is equivalent to
\begin{equation*}
\begin{split}
& \| T_1(y_1,z_1) - T_1(y_2,z_2) \|_{B,\tau_1} \leq \alpha_1 \| y_1 - y_2 \|_{B,\tau_1} \\
& \| T_2(y_1,z_1) - T_2(y_2,z_2) \|_{B,\tau_2} \leq \alpha_1 \| z_1 - z_2 \|_{B,\tau_2}
\end{split}    
\end{equation*}
For example we have that
\begin{align*}
& \| T_1(y_1,z_1) - T_1(y_2,z_2) \|_{B,\tau_1}  = \max_{x \in S_1} |T_1(y_1,z_1)(x)-T_1(y_2,z_2)(x)| e^{-\tau_1 (x-(a-\bar{a}))} \ . 
\end{align*}
Furthermore, we get that
\begin{equation*}
\begin{split}
|T_1(y_1,z_1)(x)-T_1(y_2,z_2)(x)| & \leq \int\limits_{a}^{x} | f(s,y_1(s),z_1(s)) - f(s,y_2(s),z_2(s)) | \ ds \\
& \leq L_1 \int\limits_{a}^{x} |y_1(s)-y_2(s)| \ ds \\
& = L_1 \int\limits_{a}^{x} |y_1(s)-y_2(s)| e^{-\tau_1(s-(a-\bar{a}))} e^{\tau_1(s-(a-\bar{a}))} \ ds \\
& \leq L_1 \| y_1 - y_2 \|_{B,\tau_1} \int\limits_{a}^{x} e^{\tau_1(s-(a-\bar{a}))} \ ds \\
& = L_1 \| y_1 - y_2 \|_{B,\tau_1} \cdot \dfrac{e^{\tau_1(x-(a-\bar{a}))}-1}{\tau_1} \\
& \leq \dfrac{L_1}{\tau_1} \| y_1 - y_2 \|_{B,\tau_1} e^{\tau_1(x-(a-\bar{a}))} \ .
\end{split}    
\end{equation*}
So, we obtain that
\begin{align*}
& |T_1(y_1,z_1)(x)-T_1(y_2,z_2)(x)| e^{-\tau_1(x-(a-\bar{a}))} \leq \dfrac{L_1}{\tau_1} \| y_1 - y_2 \|_{B,\tau_1} \ .
\end{align*}
Taking the maximum by $x \in S_1$, it follows that $\alpha_1 \geq \dfrac{L_1}{\tau_1}$. \\
In a similar way, for the case when we are dealing with $T_2$, it follows that
\begin{align*}
& |T_2(y_1,z_1)(x)-T_2(y_2,z_2)(x)| e^{-\tau_2(x-(a-\bar{a}))} \leq \dfrac{L_2}{\tau_2} \| z_1 - z_2 \|_{B,\tau_2} \ .
\end{align*}
So $\alpha_1 \geq \dfrac{L_2}{\tau_2}$. This means that we can choose $\tau_1,\tau_2 > 0$, such that
\begin{align*}
\alpha_1  = \max \Big\lbrace \dfrac{L_1}{\tau_1}, \dfrac{L_2}{\tau_2} \Big\rbrace < 1 \ ,   
\end{align*}
because one can see that $\rho((\alpha_1,0)) = \alpha_1 < 1$ and applying Theorem 2.9 of \cite{HuangRadenovic} with $k_2 = k_3 = k_4 = k_5 = 0$, $k_1 = \alpha_1$ and $g$ the identity mapping, then the proof is over.
\end{proof}

Finally, we present our last result of this section regarding the convergence of a sequence of solutions for a family of nonlinear differential systems. Furthermore, the following theorem is crucial, in the sense that we extend the application of S.B. Nadler Jr. from \cite{Nadler}.

\begin{theorem}\label{T4.3}
Let $D$, $\Delta$ and $I$ as in the previous theorem. Consider the following nonlinear systems of differential equations :
\begin{equation}\label{EQ4.6}
\begin{cases}
& y^\prime (x) = f_n(x,y(x),z(x)) \\
& z^\prime (x) = g_n(x,y(x),z(x)) \\
& y(a) = \beta \\
& z(a) = \gamma
\end{cases}
, \ \text{ for each } n \geq 1 \text{ and } x \in I.
\end{equation}
Furthermore, consider another system of differential equations :
\begin{equation}\label{EQ4.7}
\begin{cases}
& y^\prime (x) = f(x,y(x),z(x)) \\
& z^\prime (x) = g(x,y(x),z(x)) \\
& y(a) = \beta \\
& z(a) = \gamma
\end{cases}
\ ,
\end{equation}
where the functions $f_n,g_n,f$ and $g$ are continuous on $D$. Moreover, suppose the following assumptions are satisfied :
\begin{equation*}
\begin{split}
& (1) \text{ there exists } M \in (0,1) ,\text{ such that for all } n \in \mathbb{N} \text{ there exists } k_n, h_n, k, h \geq 0, \text{ for which one has } : \\
& \begin{cases}
& |f_n(x,y,z)-f_n(x,\bar{y},\bar{z})| \leq k_n |y-\bar{y}| \\
& |g_n(x,y,z)-g_n(x,\bar{y},\bar{z})| \leq h_n |z-\bar{z}|
\end{cases}
\ , \text{ for every } (x,y,z) \text{ and } (x,\bar{y},\bar{z}) \text{ from } D. \\
& (2) \begin{cases}
& |f(x,y,z)-f(x,\bar{y},\bar{z})| \leq k |y-\bar{y}| \\
& |g(x,y,z)-g(x,\bar{y},\bar{z})| \leq h |z-\bar{z}|
\end{cases}
\ , \text{ for every } (x,y,z) \text{ and } (x,\bar{y},\bar{z}) \text{ from } D, \\
& \ \text{ with } k_n, h_n, k, h > 0, \text{ for each } n \in \mathbb{N}, \text{ satisfying } \max \lbrace k_n, h_n, k, h \rbrace \leq M < 1. \\
& (3) \ \text{ The pointwise convergence of the families } (f_n) \text{ and } (g_n), \text{ i.e. } \\
& \begin{cases}
& f_n \xrightarrow{p} f \\
& g_n \xrightarrow{p} g 
\end{cases}
\ , \text{ i.e. } 
\begin{cases}
& \lim\limits_{n \to \infty} f_n(x,y,z) = f(x,y,z) \\
& \lim\limits_{n \to \infty} g_n(x,y,z) = g(x,y,z)
\end{cases}
\ , \text{ for every } (x,y,z) \in D. \\
& (4) \text{ If the mappings } f_n \text{ and } g_n \text{ are bounded, for each } n \in \mathbb{N} \text{ by } M_n \text{ and } \tilde{M}_n \text{ respectively,} \\
& \text{then there exist } M_f,M_g \geq 0, \text{ such that } M_n \leq M_f \text{ and } \tilde{M}_n \leq M_g, \text{ for each } n \in \mathbb{N}.
\end{split}
\end{equation*}
If $(y_n,z_n)$ is the unique solution of \ref{EQ4.6} and $(y,z)$ is the unique solution of \ref{EQ4.7}, then
\begin{equation*}
\begin{split}
\begin{cases}
& y_n \xrightarrow{u} y \\
& z_n \xrightarrow{u} z
\end{cases}
\ , \text{ i.e. }
\begin{cases}
& \lim\limits_{n \to \infty} \| y_n - y \|_{B,\tau_1} = 0 \\
& \lim\limits_{n \to \infty} \| z_n - z \|_{B,\tau_2} = 0 
\end{cases}
\ .
\end{split}    
\end{equation*}
\end{theorem}

\begin{proof}
The first order system of differential equations \ref{EQ4.6} can be written under an integral form, as follows : 
\begin{equation}\label{EQ4.8}
\begin{split}
\begin{cases}
& y (x) = \beta + \int\limits_{a}^{x} f_n(s,y(s),z(s)) \ ds \\
& z (x) = \gamma + \int\limits_{a}^{x} g_n(s,y(s),z(s)) \ ds
\end{cases}    
\ .
\end{split}
\end{equation}
Furthermore, the system \ref{EQ4.7} can be written also under an integral form, i.e.
\begin{equation}\label{EQ4.9}
\begin{split}
\begin{cases}
& y (x) = \beta + \int\limits_{a}^{x} f(s,y(s),z(s)) \ ds \\
& z (x) = \gamma + \int\limits_{a}^{x} g(s,y(s),z(s)) \ ds
\end{cases}    
\ .
\end{split}
\end{equation}
Similar to the proof of Theorem \ref{T4.2}, we define the operators $T_n = (T_{n,1},T_{n,2})$ and $T=(T_1,T_2)$, such that
\begin{equation*}
\begin{split}
\begin{cases}
& T_{n,1}(y,z)(x) = \beta + \int\limits_{a}^{x} f_n(s,y(s),z(s)) \ ds \\
& T_{n,2}(y,z)(x) = \gamma + \int\limits_{a}^{x} g_n(s,y(s),z(s)) \ ds
\end{cases}
\ \text{ and }
\begin{cases}
& T_1(y,z)(x) = \beta + \int\limits_{a}^{x} f(s,y(s),z(s)) \ ds \\
& T_2(y,z)(x) = \gamma + \int\limits_{a}^{x} g(s,y(s),z(s)) \ ds
\end{cases}
\end{split}
\end{equation*}
With the same notations as in the proof of Theorem \ref{T4.2}, we define $X_1 := C(I,\| \cdot \|_{B,\tau_1})$, $X_2 := C(I,\| \cdot \|_{B,\tau_2})$ and $X := X_1 \times X_2$, respectively. Furthermore, one can consider the compact set $\Delta$ as in the proof of the previous theorem and then there exists $M_n$ and $\tilde{M}_n$ for each $n \in \mathbb{N}$ and so, by property $(4)$ we find $M_f$ and $M_g$, such that $|f_n(x,y,z)| \leq M_n \leq M_f$ and $|g_n(x,y,z)| \leq \tilde{M}_n \leq M_g$, for every $(x,y,z) \in \Delta \subset D$. Then, we can define $h_1 := \min \Big\lbrace \bar{a}, \dfrac{\bar{\beta}}{M_f} \Big\rbrace$ and $h_2 := \min \Big\lbrace \bar{a}, \dfrac{\bar{\gamma}}{M_g} \Big\rbrace$. Furthermore, based on $h_1$ and $h_2$, one can define $S_1,S_2$ and $S$ as in the proof of the previous theorem. Then, by assumptions $(1)$ and $(2)$, applying Theorem \ref{T4.2}, we get that 
\begin{equation*}
\begin{split}
& \text{ there exists and is unique } (y_n,z_n) \text{ solution for the system } \ref{EQ4.6} \\
& \text{ there exists and is unique } (y,z) \text{ solution for the system } \ref{EQ4.7} \ ,
\end{split}    
\end{equation*}
where $y_n$ and $y$ are from $S_1$ and $z_n,z$ are from $S_2$. Also, we observe that $T_n, T : S \to S$, $T_{n,1} : S \to S_1$ and $T_{n,2} : S \to S_2$. So, taking $(y,z) \in S$ an arbitrary element (we use the same notation as the unique solution of the system \ref{EQ4.7} since it lies no confusion) and $x \in [a-h_1,a+h_1]$, for each $n \geq 1$, it follows that
\begin{align*}
\left[ T_{n,1}(y,z)(x) - T_1(y,z)(x) \right] = \int\limits_{a}^{x} \left[ f_n(t,y(t),z(t))-f(t,y(t),z(t)) \right] \ dt     
\end{align*}
Now, since $f_n \xrightarrow{p} f$ and $|f_n| \leq M_f$, by Lebesgue dominated convergence theorem, it follows that 
\begin{align*}
\lim\limits_{n \to \infty} |T_{n,1}(y,z)(x)-T_1(y,z)(x)| = 0 \ .    
\end{align*}
This is equivalent to : for every $\varepsilon_1 > 0$ and for every $x \in [a-h_1,a+h_1]$, there exists $N_1 \geq 0$, such that for all $n \geq N_1$, one has $|T_{n,1}(y,z)(x)-T_1(y,z)(x)| < \varepsilon_1$.  \\
In a similar way, for $T_{n,2}$, we get 
\begin{align*}
\left[ T_{n,2}(y,z)(x) - T_2(y,z)(x) \right] = \int\limits_{a}^{x} \left[ g_n(t,y(t),z(t))-g(t,y(t),z(t)) \right] \ dt     
\end{align*}
Now, since $g_n \xrightarrow{p} g$ and $|g_n| \leq M_g$, by Lebesgue dominated convergence theorem, it follows that 
\begin{align*}
\lim\limits_{n \to \infty} |T_{n,2}(y,z)(x)-T_2(y,z)(x)| = 0 \ .    
\end{align*}
This is equivalent to : for every $\varepsilon_2 > 0$ and for every $x \in [a-h_2,a+h_2]$, there exists $N_2 \geq 0$, such that for all $n \geq N_2$, one has $|T_{n,2}(y,z)(x)-T_2(y,z)(x)| < \varepsilon_2$. \\
This means that $T_{n,1}(y,z) \xrightarrow{p} T_1(y,z)$ and $T_{n,2}(y,z) \xrightarrow{p} T_2(y,z)$, where we have the usual pointwise convergence. Furthermore, we show that the family $(T_{n,1}(y,z))$ is uniformly equicontinuous in the classical sense, i.e. for every $\varepsilon_1 > 0$, there exists $\delta_1 = \delta_1 (\varepsilon_1) > 0$, such that for each $n \in \mathbb{N}$ and for every $x,\bar{x} \in [a-h_1,a+h_1]$ satisfying $d(x,\bar{x}) < \delta_1$, we must have that $| T_{n,1}(y,z)(x) - T_{n,1}(y,z)(\bar{x}) | < \varepsilon_1$. Moreover, since $| T_{n,1}(y,z)(x) - T_{n,1}(y,z)(\bar{x}) | \leq M_f |x-\bar{x}| < \delta_1 M_f$, we can easily choose $\delta_1 := \dfrac{\varepsilon_1}{M_f}$. \\
In a similar manner, we show that the family $(T_{n,2}(y,z))$ is also uniformly equicontinuous in the classical sense, i.e. for every $\varepsilon_2 > 0$, there exists $\delta_2 = \delta_2 (\varepsilon_2) > 0$, such that for each $n \in \mathbb{N}$ and for every $x,\bar{x} \in [a-h_2,a+h_2]$ satisfying $d(x,\bar{x}) < \delta_2$, we must have that $| T_{n,2}(y,z)(x) - T_{n,2}(y,z)(\bar{x}) | < \varepsilon_2$. Moreover, since $| T_{n,2}(y,z)(x) - T_{n,2}(y,z)(\bar{x}) | \leq M_g |x-\bar{x}| < \delta_2 M_g$, we can easily choose $\delta_2 := \dfrac{\varepsilon_2}{M_g}$. At the same time, we have the following : 
\begin{equation*}
\begin{split}
\begin{cases}
& T_{n,1}(y,z) \xrightarrow{p} T_1(y,z) \\
& (T_{n,1}(y,z)) \text{ uniformly equicontinuous, so it is also equicontinuous. }
\end{cases}
\end{split}    
\end{equation*}
and by Arzela-Ascoli theorem, we find that $T_{n,1}(y,z) \xrightarrow{u} T_1(y,z)$, for each $(y,z) \in S$ where the uniform convergence is on $I_1 := [a-h_1,a+h_1]$. \\
In a similar way, we find that
\begin{equation*}
\begin{split}
\begin{cases}
& T_{n,2}(y,z) \xrightarrow{p} T_2(y,z) \\
& (T_{n,2}(y,z)) \text{ equicontinuous }
\end{cases}
\end{split}    
\end{equation*}
and by Arzela-Ascoli theorem, we find that $T_{n,2}(y,z) \xrightarrow{u} T_2(y,z)$, for each $(y,z) \in S$ where the uniform convergence is on $I_2 := [a-h_2,a+h_2]$. \\
This equivalent to the fact that for each $(y,z) \in S$ and for $\varepsilon_1 > 0$, there exists $N_1 = N_1(\varepsilon_1,y,z) \geq 0$, such that for every $n \geq N_1$, we have that $|T_{n,1}(y,z)(x)-T_{1}(y,z)(x)| < \varepsilon_1$. At the same time $|T_{n,1}(y,z)(x)-T_{1}(y,z)(x)| < \varepsilon_1$ implies that $|T_{n,1}(y,z)(x)-T_{1}(y,z)(x)| e^{-\tau_1 (x-(a-\bar{a}))} < \varepsilon_1$. So, taking the maximum when $x \in [a-h_1,a+h_1]$, we find that 
\begin{align*}
\| T_{n,1}(y,z) - T_{1}(y,z) \|_{B,\tau_1} \leq \varepsilon_1 \ .    
\end{align*}
So, this implies that
\begin{align*}
T_{n,1} \xrightarrow{p} T_1, \text{ where the pointwise convergence is on } S \ .    
\end{align*}
In an analogous way, we have that for each $(y,z) \in S$ and for $\varepsilon_2 > 0$, there exists $N_2 = N_2(\varepsilon_2,y,z) \geq 0$, such that for every $n \geq N_2$, we have that $|T_{n,2}(y,z)(x)-T_{2}(y,z)(x)| < \varepsilon_2$. At the same time $|T_{n,2}(y,z)(x)-T_{2}(y,z)(x)| < \varepsilon_2$ implies that $|T_{n,2}(y,z)(x)-T_{2}(y,z)(x)| e^{-\tau_2 (x-(a-\bar{a}))} < \varepsilon_2$. So, taking the maximum when $x \in [a-h_2,a+h_2]$, we find that 
\begin{align*}
\| T_{n,2}(y,z) - T_{2}(y,z) \|_{B,\tau_2} \leq \varepsilon_2 \ .    
\end{align*}
So, this implies that
\begin{align*}
T_{n,2} \xrightarrow{p} T_2, \text{ where the pointwise convergence is on } S \ .    
\end{align*}
Now, it is time to show that $T_n \xrightarrow{p} T$ with respect to the Banach algebra $\mathbb{R}^2$. For example, taking $c=(\varepsilon_1,\varepsilon_2) \in \mathbb{R}^2$ arbitrary, with $\varepsilon_1,\varepsilon_2 > 0$ and taking $(y,z) \in S$ also arbitrary, then there exists $N = \max \lbrace N_1,N_2 \rbrace$ that depends on $c$, $y$ and $z$, with $N \geq 0$, such that for all $n \geq N$, we have that
\begin{equation*}
\begin{split}
\begin{cases}
& \| T_{n,1}(y,z) - T_1(y,z) \|_{B,\tau_1} \leq \varepsilon_1 \\
& \| T_{n,2}(y,z) - T_2(y,z) \|_{B,\tau_2} \leq \varepsilon_2 
\end{cases}
\ .
\end{split}    
\end{equation*}
This means that 
\begin{equation*}
\begin{split}
\begin{cases}
& \left( \| T_{n,1}(y,z) - T_1(y,z) \|_{B,\tau_1}, \| T_{n,2}(y,z) - T_2(y,z) \|_{B,\tau_2} \right) \preceq (\varepsilon_1,\varepsilon_2) \Leftrightarrow \\
& d(T_n(y,z),T(y,z)) = d((T_{n,1}(y,z),T_{n,2}(y,z)),(T_1(y,z),T_2(y,z))) \preceq c
\end{cases}
\ .
\end{split}    
\end{equation*}
This means that $T_n \xrightarrow{p} T$, where the pointwise convergence is on $S$ and is in the setting of the given Banach algebra. On the other hand, applying Theorem \ref{T4.2}, since $T_n$ and $T$ are cone self-contractions on $S$ and also applying Theorem \ref{T2.10}, we get the desired conclusion. Finally, we make the crucial remark regarding the method used in order to apply the already mentioned theorems. The idea behind it is very similar to the one used in the proof of Theorem \ref{T4.1}. For the integral operators from our theorem, we have the contraction cone elements are $\alpha_n = (\alpha_n^1,0)$ and $\alpha = (\alpha_0^1,0)$. Furthermore, following the proof of the previous theorem, they must satisfy $\alpha_n^1 = \max \lbrace \dfrac{k_n}{\tau_1}, \dfrac{h_n}{\tau_2} \rbrace < 1$ and $\alpha_0^1 = \max \lbrace \dfrac{k}{\tau_1}, \dfrac{h}{\tau_2} \rbrace < 1$, respectively. For simplicity, taking $\tau_1 < \tau_2$, we observe that we get $\max \lbrace k_n, h_n \rbrace < \tau_1$ and $\max \lbrace k,h \rbrace < \tau_1$. From our assumptions, we know that $\max \lbrace k_n,h_n,k,h \rbrace \leq M$, so we can take $\tau_1$ to be greater than the fixed positive constant $M < 1$ and now the proof is complete.
\end{proof}

%\subsection*{Acknowledgments}

%%%%%%%%%%%%%%%%%%%%%%%%%%%%%%%%%%%%%%%%%%

\end{document}